\documentclass[11pt]{article}

\usepackage[margin=1in]{geometry}
\usepackage{amsmath,amssymb,amsthm,mathtools}
\usepackage{graphicx}
\usepackage{subcaption}
\usepackage{booktabs}
\usepackage{hyperref}
\usepackage{enumitem}
\usepackage{xcolor}
\usepackage{microtype}
\usepackage{float}

\hypersetup{colorlinks=true,linkcolor=blue,citecolor=blue,urlcolor=blue}

\newtheorem{proposition}{Proposition}

\theoremstyle{definition}
\newtheorem{definition}{Definition}

\newtheorem{example}{Example}

\theoremstyle{remark}
\newtheorem{remark}{Remark}

\title{On the Sampling of Certain Dirichlet and Prime Sums at Gram Points and their Viscosity}
\author{Yochay Jerby}
\date{July 2026}

\begin{document}
\maketitle
\begin{abstract}
Consider the Dirichlet polynomial and its prime-indexed sub-sum
\[
\begin{array}{ccc}
E(t)
:=
\sum_{k\le \left[ \sqrt{ \frac{t}{2\pi} } \right]}
\frac{e^{it\log k}}{\sqrt{k}},
& ; & E_1(t)
:=
\sum_{\substack{p\le \left[ \sqrt{\frac{t}{2\pi}} \right]\\ p\ {\rm prime}}}
\frac{e^{it\log p}}{\sqrt p}.
\end{array} \]
In this work we investigate the statistical and arithmetic properties of the sets
\[
\mathcal{C}_N
:=
\left\{
E(g_n):1\le n\le N
\right\},
\qquad
\mathcal{P}_N
:=
\left\{
E_1(g_n):1\le n\le N
\right\},
\]
obtained by sampling these sums at the Gram points \(g_n\) of the Hardy \(Z\)-function. We refer to \(\mathcal{C}_N\) and \(\mathcal{P}_N\) as the full and prime \(N\)-th Gram clouds, respectively. The motivation is that the approximate functional equation expresses both \(Z(g_n)\) and \(Z'(g_n)\) through different projections of the same oscillatory phase configuration, suggesting that the geometry of these sampled sums may retain information about the local behavior of nearby zeros. Numerically, we find that Gram-point sampling produces highly non-generic geometric structures: good and isolated bad Gram points occupy markedly different regions of the clouds, the prime cloud of isolated bad points splits into two anisotropic lobes, and the position along the principal direction of the folded prime cloud is strongly correlated with the viscosity of the Hardy $Z$-function
\(
\mu(g_n)=\left|\frac{Z'(g_n)}{Z(g_n)}\right|,
\)
with smallest-viscosity points concentrating near distinguished extremal regions of the lobes.
\end{abstract}
\tableofcontents

\section{Introduction} 

\subsection{The Full and Prime Gram Clouds} 

Recall that the Riemann--Siegel theta function is
\[
\theta(t)
:=
\Im\log\Gamma\!\left(\frac14+\frac{it}{2}\right)
-\frac{t}{2}\log\pi .
\] 
By Stirling's formula,
\[
\theta(t)
=
\frac{t}{2}\log\!\left(\frac{t}{2\pi}\right)
-\frac{t}{2}
-\frac{\pi}{8}
+O\!\left(\frac1t\right),
\qquad t\to\infty.
\]
The \(n\)-th Gram point \(g_n\) is defined by
\[
\theta(g_n)=n\pi.
\]

Gram \cite{Gram1903} observed that the zeros of \(Z(t)\) tend to
alternate with the Gram points and that the signs of \(Z(t)\) at these points
usually obey the rule
\begin{equation}
\label{eq:Gram}
        (-1)^n Z(g_n)>0.
\end{equation} 
This empirical sign rule is known as Gram's law. Since \(Z(t)\) is real on the
real axis, if two consecutive Gram points \(g_n,g_{n+1}\) both satisfy Gram's
law, then the values \(Z(g_n)\) and \(Z(g_{n+1})\) have opposite signs, and
there must be a zero of \(Z(t)\) in the Gram interval \([g_n,g_{n+1}]\).  If
Gram's law were true for all \(n\), this would give an extremely simple
mechanism forcing the zeros of \(Z(t)\) to be real.

However, Gram's law is not always true.  The first violations were discovered
by Hutchinson in 1925, who found that the Gram points \(g_{126}\) and \(g_{137}\)
fail the expected sign rule \cite{Hutchinson}.  We hence refer to \(g_n\) as a
\emph{good Gram point} if it satisfies Gram's law \eqref{eq:Gram}, and as a
\emph{bad Gram point} if it fails. We refer to a bad Gram point $g_n$ as \emph{isolated} if its neighbours $g_{n \pm 1}$ are good. A modern experimental study of the relative frequency of good and bad Gram
points was carried out by Shanker~\cite{ShankerGramRatio}.

Although Gram's law \eqref{eq:Gram} is usually viewed as a numerical phenomenon, in 
\cite{JerbyDiscriminant} we proposed a variational interpretation according to which good and bad Gram points reflect the
local discriminant geometry controlling possible collisions of neighboring
zeros.  In this sense, failures of Gram's law are tied to the way the
Riemann hypothesis is encoded into the variational approach.

 Consider the Dirichlet polynomial
\[
E(t)
:=
\sum_{k\le \left[ \sqrt{ \frac{t}{2\pi} } \right]}
\frac{e^{it\log k}}{\sqrt{k}},
\]
and its prime-indexed sub-sum
\[
E_1(t)
:=
\sum_{\substack{p\le \left[ \sqrt{\frac{t}{2\pi}} \right]\\ p\ {\rm prime}}}
\frac{e^{it\log p}}{\sqrt p}.
\]
Figure \ref{fig:continuous_clouds} shows the values of $E(t)$ (left) and $E_1(t)$ (right) for
    \(
    0<t<g_{250000}     \) in the complex plane:
    
\begin{figure}[htbp]
    \centering
    \includegraphics[width=0.65 \textwidth]{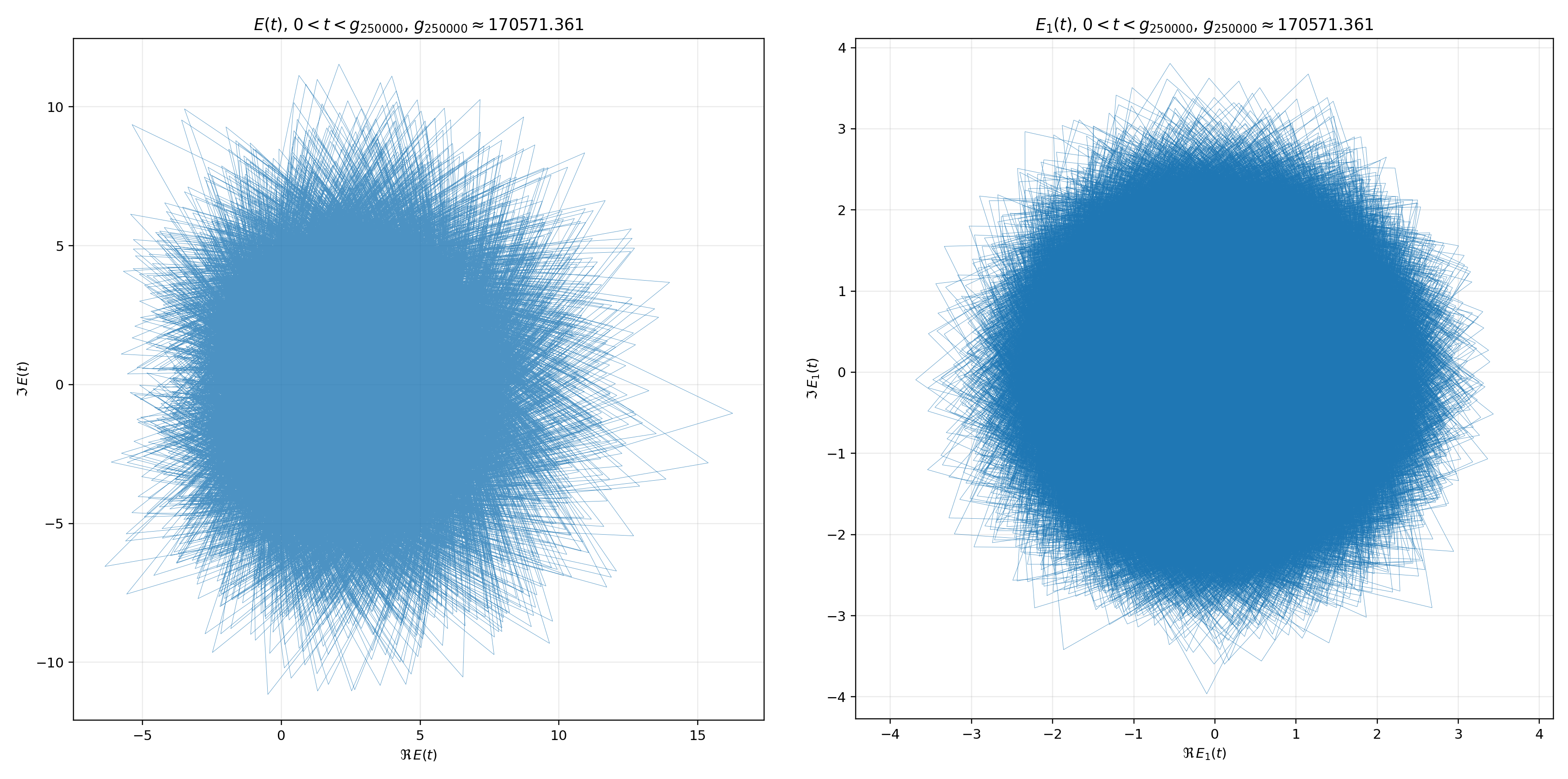}
    \caption{the values of $E(t)$ (left) and $E_1(t)$ (right) for
    \(
    0<t<g_{250000} \approx 170571.36 \)}
    \label{fig:continuous_clouds}
\end{figure}

We have:

\begin{definition}[Full and prime Gram clouds]
For \(N\in\mathbb{N}\), define
\begin{equation}
\label{eq:clouds}
\mathcal{C}_N
:=
\left\{
E(g_n):1\le n\le N
\right\},
\qquad
\mathcal{P}_N
:=
\left\{
E_1(g_n):1\le n\le N
\right\}.
\end{equation}
These sets are obtained by sampling \(E(t)\) and \(E_1(t)\), respectively,
at the Gram points \(g_n\) of the Hardy \(Z\)-function. We refer to
\(\mathcal{C}_N\) and \(\mathcal{P}_N\) as the full and prime \(N\)-th Gram
clouds, respectively.
\end{definition}

 The aim of this work is to investigate the statistical and arithmetic properties of the full and prime cloud sets \(\mathcal{C}_N\) and \(\mathcal{P}_N\).   Figure \ref{fig:Gram_clouds} compares the full AFE cloud with its
prime-indexed component for \(0<t<g_{250000}\).  

\begin{figure}[htbp]
   \centering
   \includegraphics[width=0.65 \textwidth]{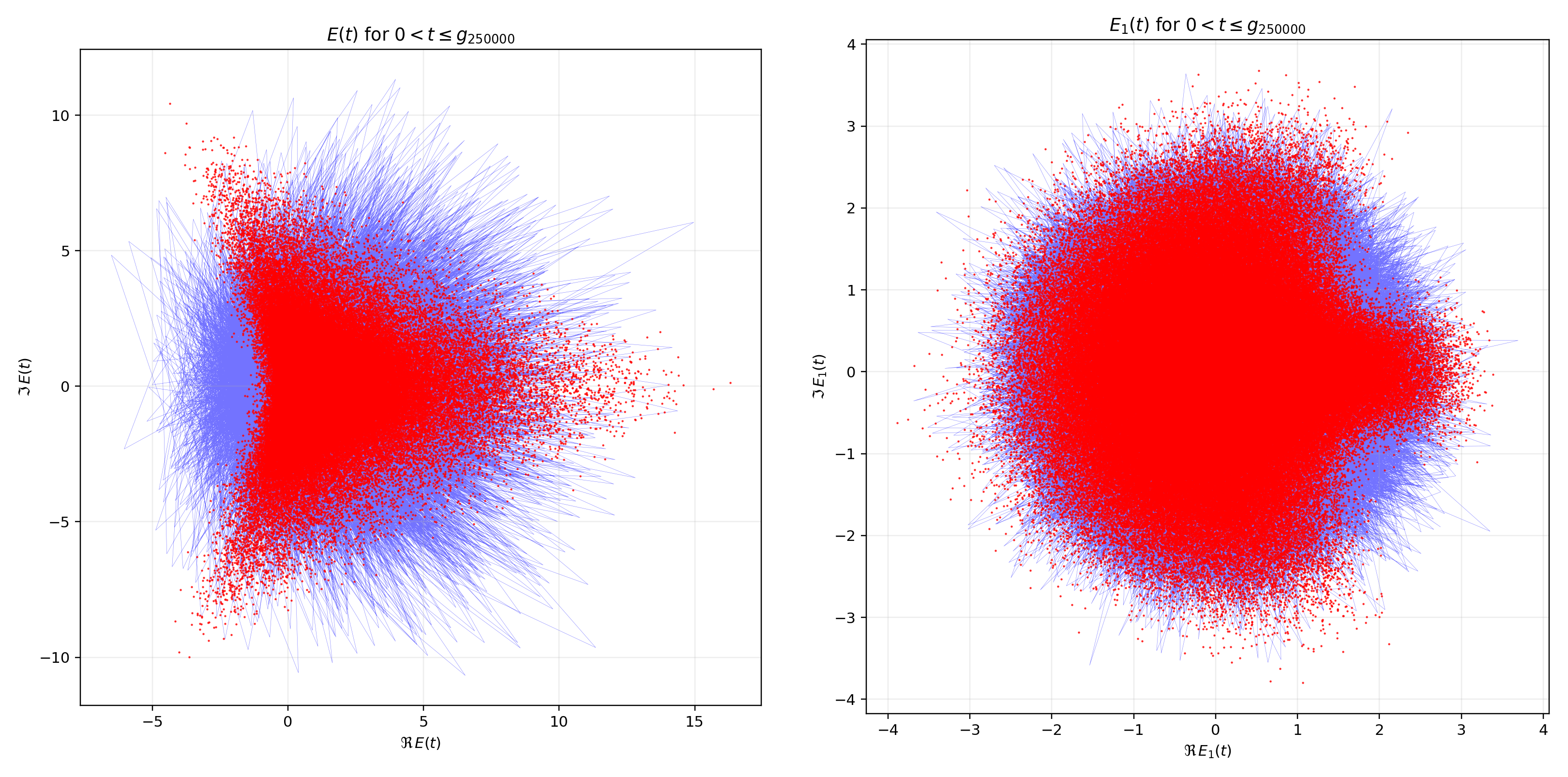}
   \caption{The full cloud and its prime component for
   \(0<t<g_{250000}\).  The left panel shows \(E(t)\) and the full cloud $\mathcal{C}_N$, while
   the right panel shows \(E_1(t)\) and the prime cloud $\mathcal{P}_N$.  Blue denotes values
   over the sampled \(t\)-range, and red denotes the clouds of Gram
   points samplings at $g_n$.}
   \label{fig:Gram_clouds}
\end{figure}

In each panel, the blue
cloud represents the values obtained over the sampled \(t\)-range, while the
red points indicate the corresponding values at the Gram points \(t=g_n\).
The left panel shows the full sum \(E(t)\) together with $\mathcal{C}_N$, and the right panel its prime
component \(E_1(t)\) together with $\mathcal{P}_N$.

Figure \ref{fig:Gram_clouds} shows that sampling at the Gram points is far
from neutral and that the sets \(\mathcal C_N\) and \(\mathcal P_N\) do not simply
reproduce the ambient clouds traced by \(E(t)\) and \(E_1(t)\) as \(t\)
varies continuously.  Rather, the Gram-point samples concentrate on highly
structured subsets of those clouds, with a pronounced anisotropy and
directional organization that is especially visible in the prime component.
Thus the condition $t=g_n$
appears to impose a strong geometric selection on the underlying phase
configuration.

\subsection{The Split of the Full and Prime Clouds to their Good and Bad Parts} 

Recall that the value of \(Z(g_n)\) and $Z'(g_n)$ at a Gram point $g_n$ can be evaluated using the approximate functional equation (AFE). Concretely, the formula gives 
\begin{equation}
\label{eq:afe_value}
        Z(g_n)
        =
        2(-1)^n
        \sum_{k\leq N(g_n)}
        \frac{\cos(g_n\log k)}{\sqrt{k}}
        +
        O\!\left(g_n^{-1/4}\right).
\end{equation}
where \(        N(t)
        =
        \left\lfloor
        \sqrt{\frac{t}{2\pi}}
        \right\rfloor 
\). The derivative is obtained from the same phase vector, but with the
additional logarithmic weight
\(\log\!\left(g_n/(2\pi k^2)\right)\):
\[
        Z'(g_n)
        =
        (-1)^n
        \Im\left(
        \sum_{k\leq N(g_n)}
        \log\!\left(\frac{g_n}{2\pi k^2}\right)
        \frac{e^{ig_n\log k}}{\sqrt{k}}
        \right)
        +
        O\!\left(g_n^{-1/4}\log^2 g_n\right),
\]
see Eq.~(4.5) and Eq.~(5.7) of \cite{IvicZ}. 
Thus, \(E(g_n)\) records the underlying complex configuration of the
oscillatory phases appearing simultaneously in the AFE expressions for
\(Z(g_n)\) and \(Z'(g_n)\). The value \(Z(g_n)\) is governed by its
unweighted real projection, whereas the derivative $Z'(g_n)$ is governed by a
logarithmically weighted imaginary projection. Figure~\ref{fig:afe_cloud} displays the full cloud $\mathcal{C}_N$ of \(E(g_n)\) for the \(N=250{,}000\)
        consecutive Gram points \(1\leq n\leq250000\) with Good Gram points shown in blue and isolated bad Gram points are shown in red 

\begin{figure}[htbp]
        \centering
        \includegraphics[width=0.5 \textwidth]{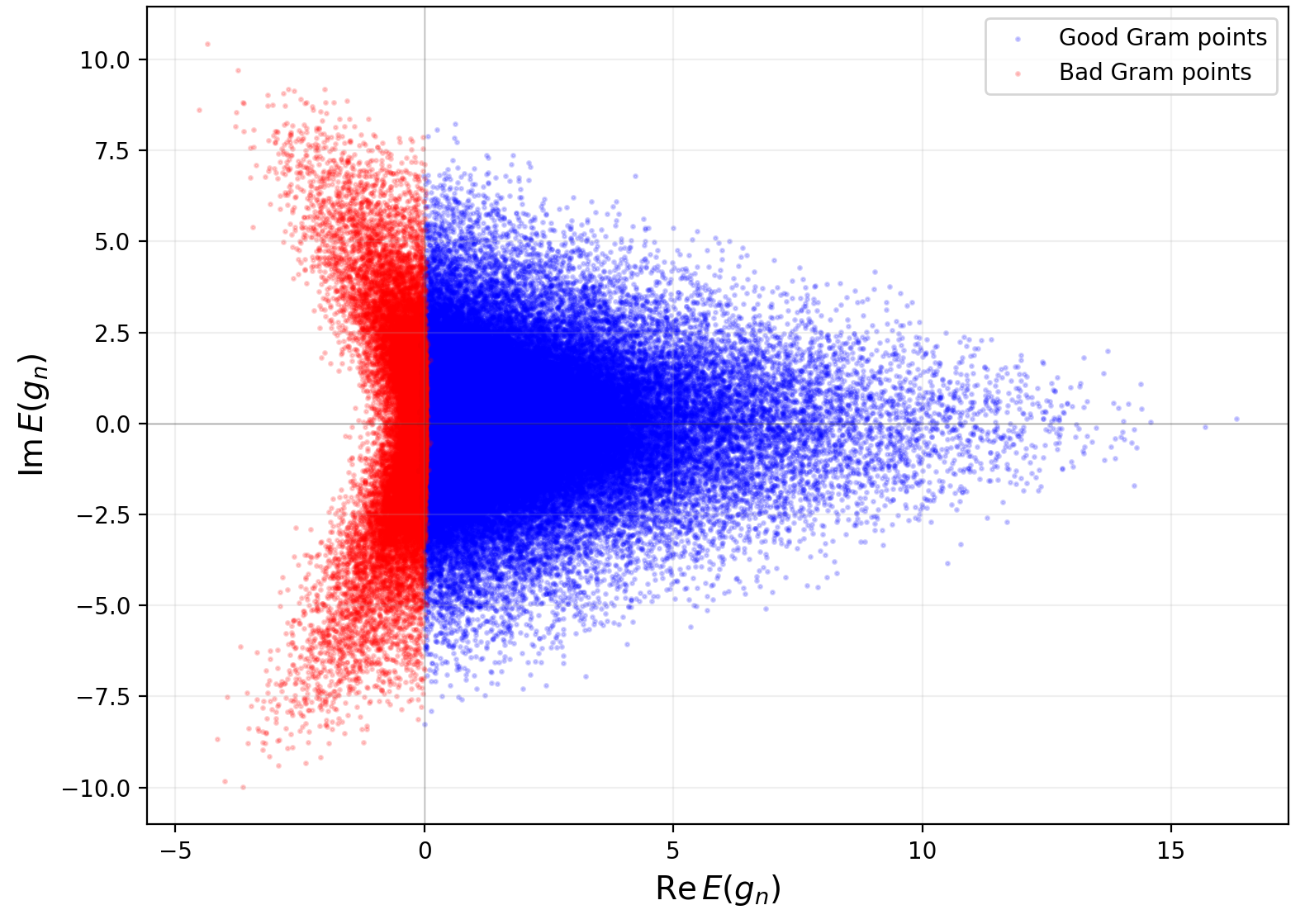}
        \caption{The full cloud $\mathcal{C}_N$ of \(E(g_n)\) for the \(N=250{,}000\)
        consecutive Gram points \(1\leq n\leq250000\). Good Gram points are
        shown in blue and bad Gram points in red. }
        \label{fig:afe_cloud}
\end{figure}

The horizontal separation between good and bad Gram points, visible in Figure~\ref{fig:afe_cloud}, follows from the fact that from 
\eqref{eq:afe_value}, up to the AFE error, one has 
\[
        Z(g_n)
        =
        2(-1)^n\Re E(g_n)
        +
        O\!\left(g_n^{-1/4}\right). \]
        In particular, $\Re E(g_n)>0$ for good Gram points and $\Re E(g_n)<0$ for bad Gram points. 

Figure~\ref{fig:bad-good_clouds}
 displays the prime cloud $\mathcal{P}_N$ of \(E_1(g_n)\) for the \(N=250{,}000\)
        consecutive Gram points \(1\leq n\leq250000\) with Good Gram points shown in blue and isolated bad Gram points are shown in red 

\begin{figure}[htbp]
   \centering
   \includegraphics[width=0.5\textwidth]{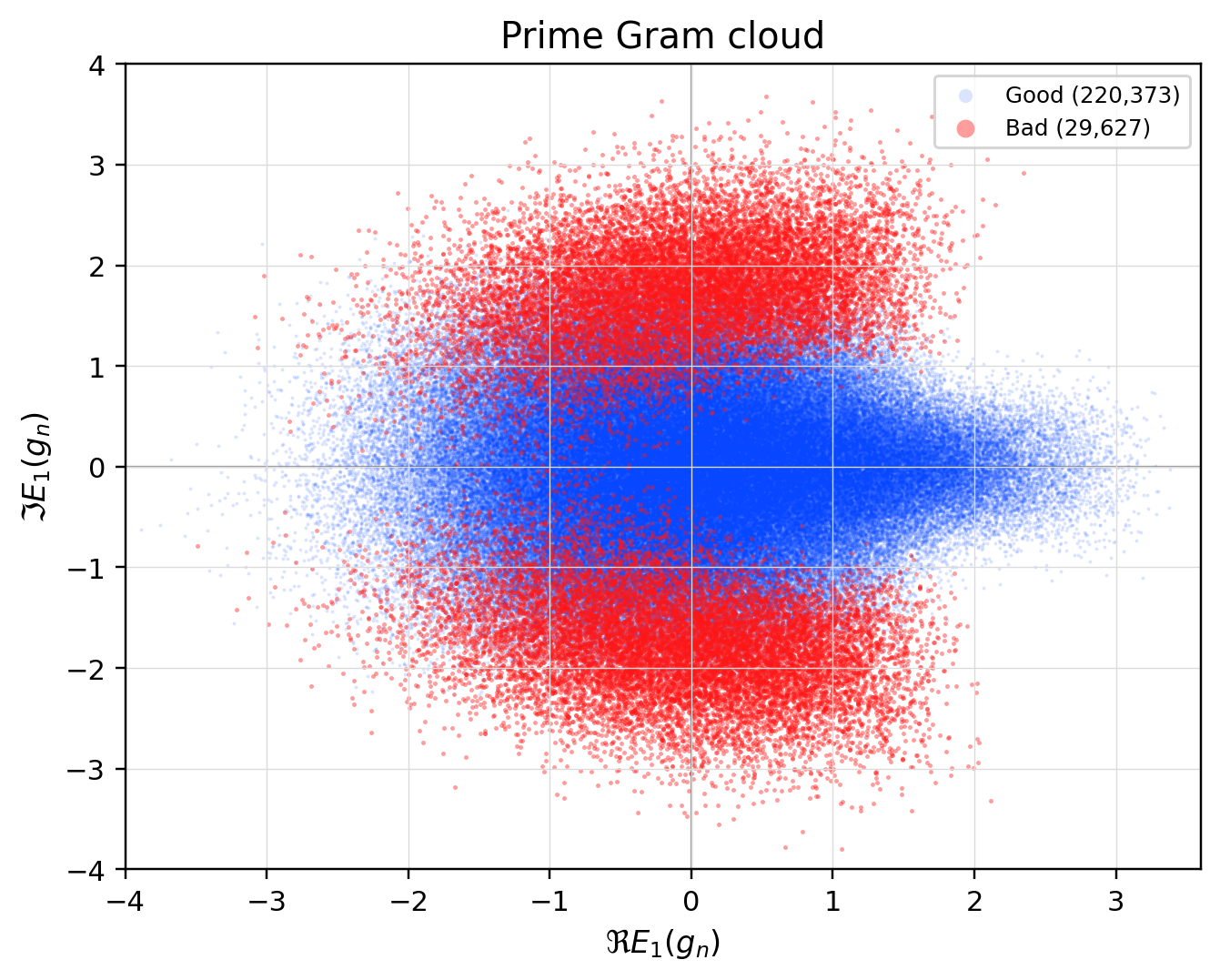}
   \caption{The prime Gram cloud
   \(\mathcal P_N=\{E_1(g_n):1\le n\le 250000\}\), with points colored
   according to Gram's law.  Good Gram points are shown in blue and bad
   Gram points in red.  The bad points separate into two pronounced lobes
   above and below the main good cloud.}
   \label{fig:bad-good_clouds}
\end{figure}

This figure is considerably more striking than the corresponding picture
for the full cloud \(\mathcal C_N\). As mentioned, in the full cloud, the separation
between good and bad Gram points is essentially forced by the definition of $E(g_n)$ and the AFE.  By
contrast, goodness is not defined in terms of the prime sum \(E_1(g_n)\).
Nevertheless, the bad points remain strongly organized after passing to
this prime projection. Instead of being randomly mixed with the good points,
they form two visible lobes around the central blue cloud.

\subsection{The Bad Prime Lobes and Viscosity}

In \cite{JerbyDiscriminant}, using the discriminant theory for deformations of
the Hardy \(Z\)-function developed there, we proposed that Gram's law and its
failures reflect an underlying local mechanism governing the motion and
collision of neighboring zeros, rather than an accidental numerical pattern. To quantify this mechanism, we introduced the local
invariant
\begin{equation}
\label{eq:visc}
        \mu(g_n)
        :=
        \left|
        \frac{Z'(g_n)}{Z(g_n)}
        \right|,
\end{equation}
which we refer to as the \emph{viscosity} of the Gram point \(g_n\). 
Initial experiments in \cite{JerbyDiscriminant} suggest that isolated bad Gram points satisfy the following repulsion property 
\begin{equation}
\mu(g_n)>4,
\end{equation} to height of up to $g_n$ of index $n \sim 2.5 \cdot 10^{7}$.  
Figure~\ref{fig:viscosity_distributions} compares the empirical distributions
of \(\ln\mu(g_n)\) for good Gram points (blue) and isolated bad Gram points (red) for the first $250,000$ Gram points.

\begin{figure}[htbp]
        \centering
        \includegraphics[width=0.6 \textwidth]{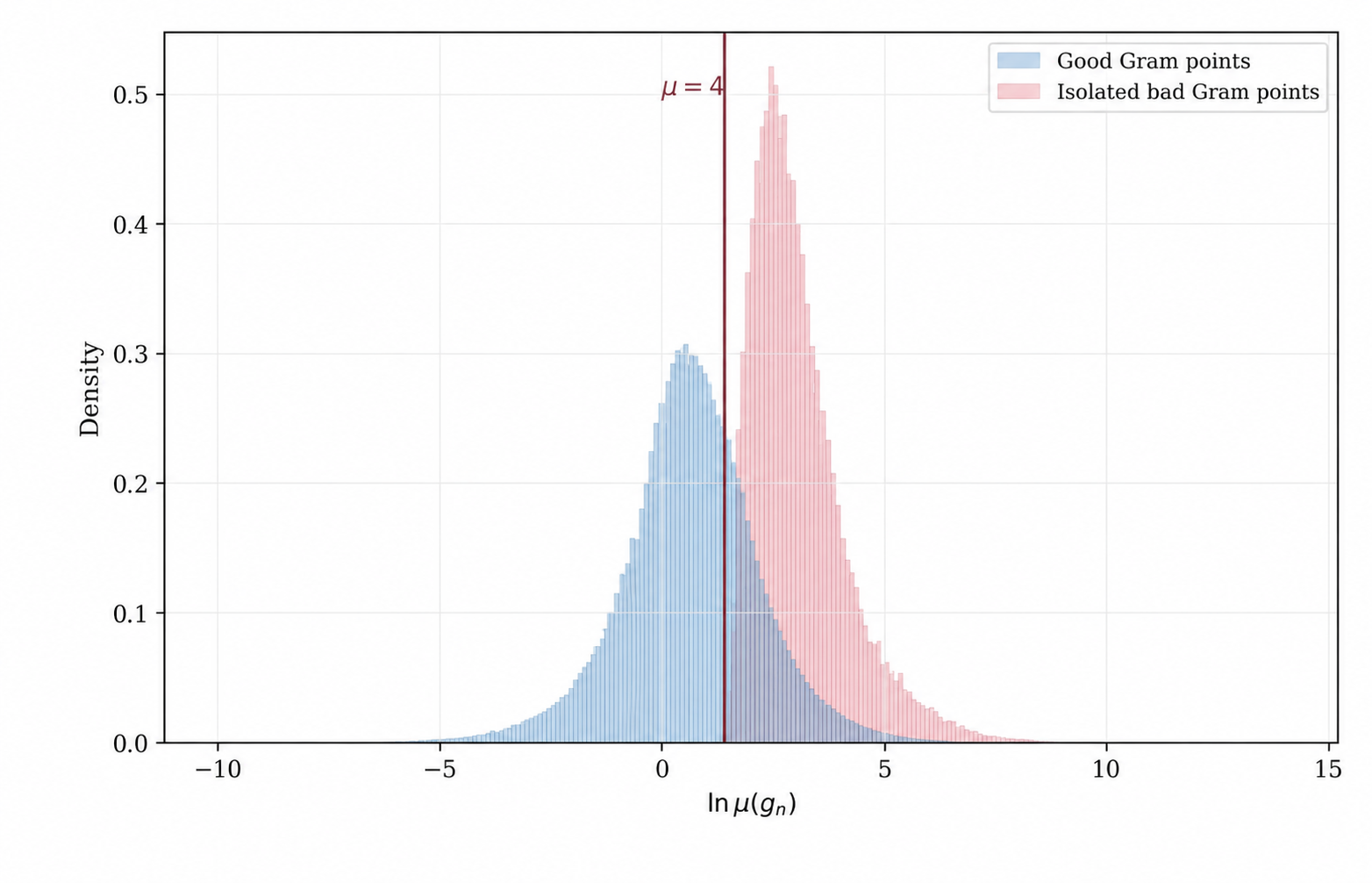}
        \caption{Empirical distributions of \(\ln\mu(g_n)\) for good Gram
        points (blue) and isolated bad Gram points (red) among the first \(250{,}000\)
        Gram points.  The vertical line marks the value \(\mu=4\).  The isolated-bad distribution exhibits
        a sharp lower boundary to the right of this line.}
        \label{fig:viscosity_distributions}
\end{figure}

The figure is based on the database of the first \(250{,}000\) Gram points,
containing \(220{,}373\) good Gram points and \(24{,}269\) isolated bad Gram
points. The distribution associated with good Gram points is
broad and approximately symmetric on the logarithmic scale, with substantial
tails extending in both directions.  By contrast, the distribution associated
with isolated bad Gram points is strongly asymmetric. It has a sharp lower edge
immediately above the repulsion bound \(\mu=4\), rises rapidly to a pronounced peak, and then
decays through a long right-hand tail. Thus the numerical evidence points not only to a lower repulsion barrier,
but to a more rigid statistical organization of isolated bad Gram points
according to the size of \(\mu(g_n)\). In the present
data set, the smallest observed value at an isolated bad Gram point is
\[
        \mu(g_n)\approx 4.26178.
\]
This suggests that the prime-indexed part of the AFE retains information
about the local Gram-point configuration.  We now investigate this phenomenon
through the viscosity
\[
        \mu(g_n)
        =
        \left|\frac{Z'(g_n)}{Z(g_n)}\right|.
\]
We shall compare the behavior of \(\mu(g_n)\) with the geometry of both the
full Gram cloud
$ \mathcal C_N$
and the prime Gram cloud
$\mathcal P_N$. For the full cloud \(\mathcal C_N\), some relation with \(\mu(g_n)\) is to
be expected.  Indeed, the AFE expresses \(Z(g_n)\), up to a small error,
through the real projection of \(E(g_n)\), while \(Z'(g_n)\) is obtained
from the same underlying phase vector with an additional logarithmic weight.
Thus the full cloud already contains much of the phase information entering
the quotient \(Z'(g_n)/Z(g_n)\).

The surprising point is that a strong geometric relation with \(\mu(g_n)\)
also appears after passing to the prime projection \(\mathcal P_N\).  In
particular, we observe that the 
position of a bad Gram point inside the bad lobes observed in \ref{fig:bad-good_clouds} is strongly correlated with the value of
\(\mu(g_n)\).  This connection between viscosity, the geometry of the prime
bad lobes, and the nearby zero configuration is one of the central themes of
the paper. 

Finally, we also compare the observed prime-cloud phenomena with the Davenport--Heilbronn setting, which provides a natural control case since it lacks the Euler-product structure of \(\zeta(s)\), violates RH, and is shown to exhibit completely different zero statistics of its full and prime Gram clouds, as well.

\bigskip

The rest of this paper is organized as follows.
Section~\ref{sec:1} presents the variational motivation for viscosity
repulsion and the Davenport--Heilbronn comparison.
Section~\ref{sec:clouds} studies the multiplicative layers of the AFE
cloud and their geometry and correlations with viscosity.
Section~\ref{sec:4} relates the bad prime lobes to nearby zero
configurations and compares these observations with the
Davenport--Heilbronn setting.
Section~\ref{sec:mechanisms} discusses the theoretical connections
through the logarithmic derivative and weighted prime sums.
We conclude with open questions and further directions in Section \ref{s:summary}.

\bigskip

\begin{remark}[Numerical database]
The complete numerical database, together with the data sets, source
code, computational scripts, animations, and supplementary material underlying
the numerical experiments in this paper, is available in the accompanying
Zenodo repository~\cite{JerbyZenodoGramPoints}.
\end{remark}

\section{Motivation for the study of the Repulsion Relation}
\label{sec:1}

 \subsection{Montgomery's conjecture and statistical repulsion of zeros}

Montgomery's conjecture  \cite{Montgomery} studies the relative
spacing between pairs of zeros at large height.  After normalizing distances by
the mean zero spacing, it predicts that the two-point density is
\[
        g(u)
        =
        1-\left(\frac{\sin \pi u}{\pi u}\right)^2 .
\]
This is the same density which governs eigenvalue spacings in the Gaussian
Unitary Ensemble.  An important qualitative implication of the conjecture is that near
the origin one has
\[
        g(u)\sim \frac{\pi^2}{3}u^2
        \qquad (u\to 0).
\]
Thus pairs of zeros with very small normalized separation are predicted to occur
with vanishing density. In other words, the limiting statistics do not allow
zeros to accumulate near one another as freely as independent random points
would.  The closer two zeros are, the more strongly their occurrence is
suppressed by the factor \(g(u)\).

In this sense Montgomery's conjecture is considered to imply a repulsive
structure among the zeros.  The term repulsion should be understood statistically as the conjecture 
does not identify a concrete dynamic force or an explicit motion between the zeros.
Rather, it means that the zero set, when viewed at large height and at the
scale of the mean spacing, is anticipated to behave like a system whose points avoid very small
gaps.  Thus Montgomery's conjecture predicts repulsion as a global statistical
property of the zeros, but not as a local dynamical
mechanism producing this avoidance.

\subsection{Statistical vs. Dynamical Repulsion of Confined Zeros}

The variation framework introduced in
\cite{JerbyAFE,JerbySections,JerbyPCC,JerbyDiscriminant}
 provides a setting in
which such a dynamical interpretation becomes possible. Consider the
parameter space
\[
    \mathcal Z_N
    =
    \left\{
    Z_N(t;\mathbf a)
    =
    Z_0(t)
    +
    \sum_{k=1}^{N}
    \frac{a_k}{\sqrt{k+1}}
    \cos\!\bigl(\theta(t)-t\log(k+1)\bigr)
    :
    \mathbf a=(a_1,\ldots,a_N)\in\mathbb R^N
    \right\},
\]
where
\[
        Z_0(t)=\cos\theta(t)
\]
is the core function. The origin \( \mathbf a=0 \) therefore corresponds to
the core, whereas \( \mathbf a=\mathbf 1=(1,\ldots,1) \) gives the arithmetic
Hardy section, which satisfies the approximation 
\[ Z(t)= Z_N(t) + O \left ( t^{-\frac{1}{4}} \right) \] in the range $2N \leq t \leq 2N+2 $. We consider a continuous path
\[
        \gamma:[0,1]\longrightarrow\mathbb R^N,
        \qquad
        \gamma(0)=0,
        \qquad
        \gamma(1)=\mathbf 1,
\]
and follow the zeros of
\(
        Z_N(t;\gamma(a))
\)
as \( a \) varies. Associated with a Gram point \( g_n \), the core function
has two neighboring zeros
\[
        t_-(0)<g_n<t_+(0),
\]
determined by
\[
        \theta\bigl(t_-(0)\bigr)
        =
        \left(n-\frac12\right)\pi,
        \qquad
        \theta\bigl(t_+(0)\bigr)
        =
        \left(n+\frac12\right)\pi.
\]
Under variation in parameter space these zeros continue locally as branches
\(
        t_-(a), t_+(a)
\). 

\begin{definition}
We will refer to the pair \( t_-(a),t_+(a) \) associated with an isolated Gram point \( g_n \) as a \emph{confined} pair of zeros if, throughout the deformation \( 0\leq a\leq1 \), these
two branches remain the only zeros of \( Z(t; \gamma(a) ) \), counted with multiplicity,
in the vertical Gram strip
\[
       g_{n-1} < Re(t) \leq g_{n+1}.
\]
\end{definition}
The problem is
then reduced to a genuine two-zero dynamical system. Starting from two
distinct real zeros, the only way for such a confined pair to leave the real
axis is through a collision
\[
        t_-(a_*)=t_+(a_*)
\]
at some parameter value \( a_* \). At this point the two zeros form a
multiple zero and may subsequently leave the real axis as a
complex-conjugate pair. This collision is detected by the Gram discriminant introduced in
\cite{JerbyDiscriminant}. Let \( g_n(\mathbf a) \) denote the extremum of
\( Z_N(t;\mathbf a) \) continuing the core Gram point \( g_n \), and define
\[
        \Delta_n(\mathbf a)
        :=
        Z_N\!\left(g_n(\mathbf a);\mathbf a\right).
\]
A collision of \( t_-(a) \) and \( t_+(a) \) occurs precisely when
\(
        \Delta_n(\gamma(a))=0.
\)
Consequently, for a confined pair the question of whether the two zeros
remain real is reduced locally to the problem of avoiding the discriminant
hypersurface
\(
        \Delta_n(\mathbf a)=0.
\)
The classical Gram law appears naturally as the first-order approximation to
this collision problem. Along the diagonal deformation
\(
        \mathbf a=r\mathbf 1 \) for \( 0\leq r\leq1
\),
one has
\[
        \Delta_n(r)
        =
        Z_N(g_n;r)+O(r^2).
\]
Thus the value \( Z(g_n) \) represents, at first order, the tendency of the
local discriminant. For a good Gram point this first-order term already has
the correct sign. For a bad Gram point it has the wrong sign and therefore
indicates a tendency toward a change of sign of the discriminant, and hence
toward a collision of the corresponding pair.

The nonlinear response provides a competing effect. The second-order
calculation of \cite{JerbyDiscriminant} gives
\[
        H_n(0)
        =
        2(-1)^n
        \left(
        \frac{Z'(g_n)}
        {\log(g_n/2\pi)}
        \right)^2 ,
\]
while \( Z'(g_n) \) also measures the infinitesimal displacement of the
continued extremum \( g_n(\mathbf a) \) in the \( t \)-direction. Thus
\( Z(g_n) \) and \( Z'(g_n) \) encode two competing local effects:
\( Z(g_n) \) measures the first-order tendency toward the collision
hypersurface, whereas \( Z'(g_n) \) measures the available lateral motion
which can move the configuration away from it.

\begin{figure}[htbp]
        \centering
        \includegraphics[width=0.75 \textwidth]{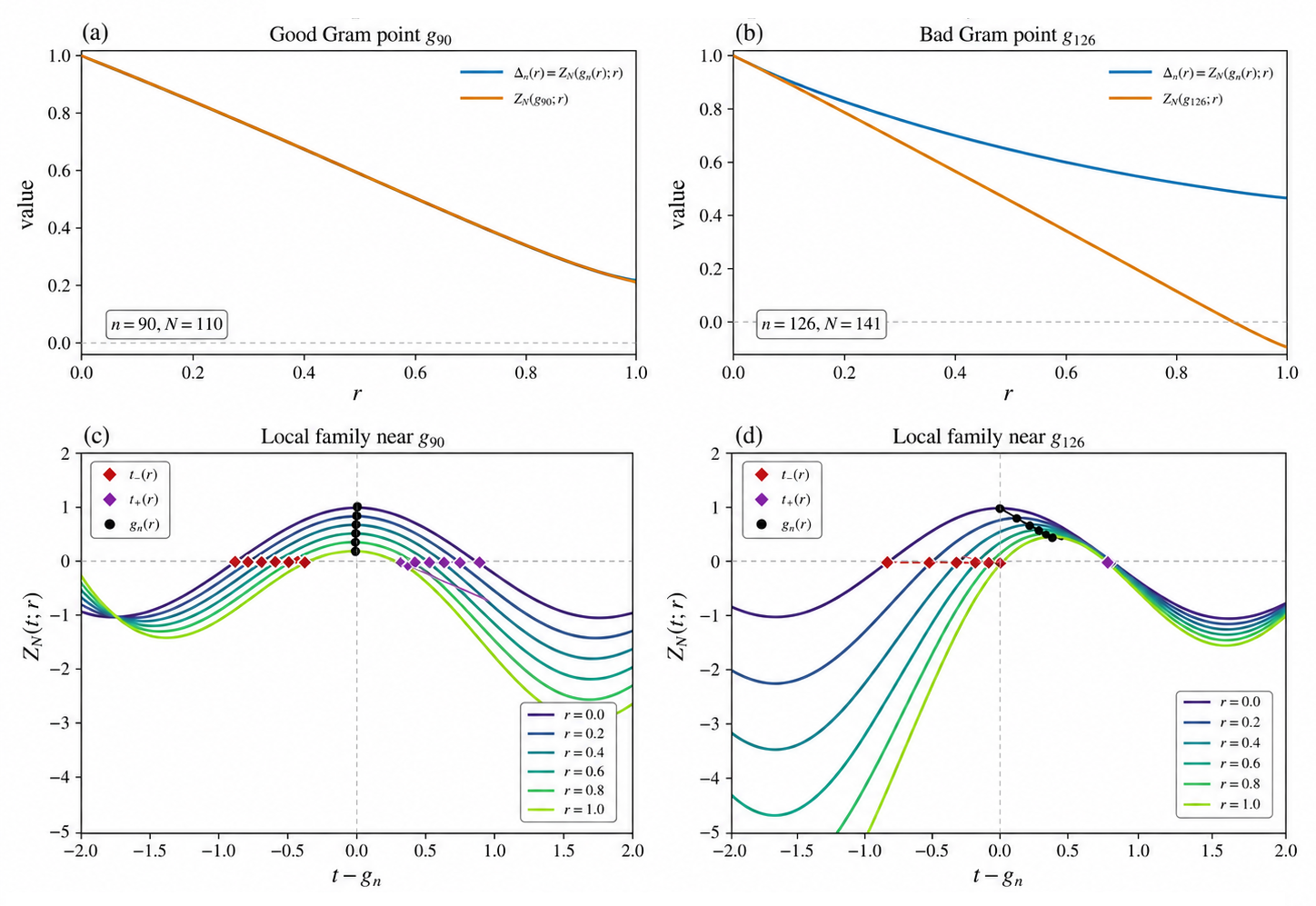}
        \caption{Local deformation at a good Gram point \(g_{90}\) and an isolated
bad Gram point \(g_{126}\). The upper panels compare the corrected
discriminant \(\Delta_n(r)\) with its fixed-point first-order approximation
\(Z_N(g_n;r)\). The lower panels show the corresponding deformed profiles
\(Z_N(t;r)\) for \(0\leq r\leq1\), together with the motion of the continued
extremum \(g_n(r)\) and the neighbouring zeros \(t_-(r)\) and \(t_+(r)\).}
        \label{fig:local_deformation_examples}
\end{figure}

Figure~\ref{fig:local_deformation_examples} illustrates this competition in two representative local configurations. In the good Gram point case, the deformation is comparatively rigid: the neighbouring zeros \(t_-(r)\) and \(t_+(r)\) remain well separated, the continued extremum \(g_n(r)\) stays essentially centred between them, and the first-order behaviour already gives a satisfactory description of the local motion. By contrast, at an isolated bad Gram point the first-order variation points much more strongly toward a possible collision of the confined pair. Nevertheless, the figure also shows the compensating nonlinear response: as the deformation develops, the extremum \(g_n(r)\) acquires a pronounced lateral displacement, and the local profile bends in a way that keeps the two zero branches apart. In this sense the bad Gram point is precisely the setting in which the tension between the collision tendency measured by \(Z(g_n)\) and the transverse escape mechanism measured by \(Z'(g_n)\) becomes visible. 
The corresponding animations, showing the dynamical version of  Fig.~\ref{fig:local_deformation_examples}, are available in the
accompanying Zenodo repository~\cite{JerbyZenodoGramPoints}. 

This motivates the definition \eqref{eq:visc} of viscosity
\(
        \mu(g_n)
\)
as a measure of this infinitesimal tradeoff. A lower bound for
\( \mu(g_n) \) may therefore be interpreted as a dynamical repulsion law:
whenever the first-order variation drives a confined pair toward collision,
the accompanying transverse motion is forced to be sufficiently strong to
resist that tendency.

In this sense, the repulsion considered here differs fundamentally from the
statistical repulsion appearing in Montgomery's conjecture. Montgomery's
pair correlation describes the scarcity of very small zero spacings in the
final zero set, whereas the viscosity describes a local dynamical mechanism
acting on a particular pair of zeros as the function is varied in
\( \mathcal Z_N \). The importance of this mechanism is also directly related to the Riemann
Hypothesis. For a confined pair, preservation of realness is equivalent to
collision avoidance. Hence a local mechanism which prevents the collision of
\( t_-(a) \) and \( t_+(a) \) is simultaneously a mechanism which prevents
these zeros from leaving the real axis. The repulsion bound may therefore
be viewed as an infinitesimal expression of a local mechanism supporting the
real-zero structure of the Hardy \( Z \)-function, see 
\cite{JerbyDiscriminant}.

\begin{remark}[Confined zeros and Rosser's law]
The confinement condition is stronger than requiring \( g_n \) to be an
isolated bad Gram point. Recall that a \emph{Gram block} of length \( N \)
is a sequence
\[
g_m,g_{m+1},\ldots,g_{m+N},
\]
whose endpoints are good and whose intermediate Gram points are bad
\cite{Edwards,RosserYoheSchoenfeld}. Thus an isolated bad Gram point
\( g_n \) is the unique interior point of a Gram block of length two,
\(
(g_{n-1},g_{n+1}]
\). Rosser's rule asserts that a Gram block of length \( N \) contains exactly
\( N \) zeros of \( Z(t) \) on the real axis. Hence a length-two block
normally contains two zeros. However, Rosser's rule is not universal, and
Trudgian proved that it fails a positive proportion of the time
\cite{Trudgian}. Thus isolated badness does not, in general, guarantee a
two-zero local configuration. It should be noted, however, that among the first \(250{,}000\) Gram points
in the data set considered, no violations of Rosser's rule occur.
Consequently, every isolated bad Gram point encountered in this range is
associated with a confined pair.

It should be noted that in our previous work~\cite{JerbyDiscriminant},
the viscosity bound
\(
        \mu(g_n)>4
\)
was investigated at substantially larger indices, including a systematic
computation of all bad Gram points up to the range
\( n\leq 2.5 \times10^{7}\).
The numerical evidence revealed a sharp distinction between isolated and
non-isolated bad Gram points. The bound was seen to persist for the isolated 
configurations, whereas bad Gram points belonging to longer Gram blocks
could have viscosity far below \(4\).  In particular, the low-viscosity
``corrupt'' points observed there were non-isolated. An extreme example is the non-isolated
\(g_{9807962}\), for which
\(\mu(g_{9807962})=0.0750883\).
Finally, we note that the data suggest more than the lower bound
\(\mu(g_n)>4\) for isolated bad Gram points associated with a confined pair
of zeros.  They also suggest the asymptotic relation
\begin{equation}
\label{eq:liminf_scaled_mu}
\liminf_{\substack{n\to\infty\\
g_n\ \mathrm{isolated\ bad,\ confined}}}
\frac{2\pi}{\log(g_n/2\pi)}\,\mu(g_n)
=
2.
\end{equation}
\end{remark}

\subsection{A non-trivial test case: The Davenport--Heilbronn function}
\label{ss:2.3}
At first sight, the restriction to confined zero pairs might appear to make
the repulsion conjecture artificially simple. Indeed, confinement isolates
exactly two zeros inside a fixed Gram strip and excludes the possibility that
additional zeros enter the local configuration. Assuming the Riemann-Hypothesis, one might therefore suspect
that the real-zero behaviour is already largely contained in the confinement
assumption itself. The Davenport--Heilbronn function provides a particularly
important counterexample to this point-of-view. Recall the Davenport--Heilbronn function
\[
        D(s)
        =
        \frac{1-i\kappa}{2}L(s,\chi_{5,2})
        +
        \frac{1+i\kappa}{2}L(s,\overline{\chi}_{5,2}),
\]
where \( \chi_{5,2} \) denotes the non-real primitive Dirichlet character modulo \(5\), chosen so that \( \chi_{5,2}(2)=i \) and
\[
        \kappa
        =
        \frac{\sqrt{10-2\sqrt5}-2}{\sqrt5-1}.
\]
Its completed form
\[
        \xi_{DH}(s)
        =
        \left(\frac{\pi}{5}\right)^{-s/2}
        \Gamma\!\left(\frac{1+s}{2}\right)D(s)
\]
satisfies the symmetric functional equation
\[
        \xi_{DH}(s)=\xi_{DH}(1-s).
\]

Thus the Davenport--Heilbronn function shares the same basic reflection
symmetry that underlies the construction of the Hardy \( Z \)-function,
yet, in sharp contrast to the Riemann zeta function, it is known to
possess zeros off the critical line. In particular, it provides a
classical example of a zeta-like function for which the analogue of the
Riemann Hypothesis fails, despite the presence of the same formal
symmetry coming from the functional equation
\cite{DavenportHeilbronnI,DavenportHeilbronnII,Titchmarsh}.

The corresponding variation framework may be constructed in direct analogy
with the space \( \mathcal Z_N \). Its core is obtained from
\[
        D_0(s)
        =
        \frac12
        \left[
        \left(\frac{\pi}{5}\right)^{-s/2}
        \Gamma\!\left(\frac{1+s}{2}\right)
        +
        \left(\frac{\pi}{5}\right)^{(s-1)/2}
        \Gamma\!\left(\frac{2-s}{2}\right)
        \right],
\]
and, after normalization on the critical line, gives a real core function
\( Z_0^{DH}(t) \). As in the \( Z \)-function case, the arithmetic terms are
then activated through a parameter space
\[
        \mathcal Z_N^{DH}
        =
        \bigl\{
        Z_N^{DH}(t;\mathbf a)
        \bigr\},
\]
with the core at \( \mathbf a=0 \) and the arithmetic Davenport--Heilbronn
section at \( \mathbf a=\mathbf 1 \). Along the diagonal deformation
\(\mathbf a=r\mathbf 1 \), for
$0\leq r\leq1$,
one obtains a family
\(
        Z_N^{DH}(t;r)
\)
and a continued extremum \( g_n^{DH}(r) \), exactly as for the Hardy
\( Z \)-function.

The first pair of zeros responsible for the failure of the
Riemann-hypothesis analogue occurs near the Gram point \( g_{44}^{DH} \).
At the core, the corresponding zeros
\[
        t_-^{DH}(0)
        <
        g_{44}^{DH}
        <
        t_+^{DH}(0)
\]
are distinct and real. Moreover, throughout the relevant deformation they
form a confined local pair. The phenomenon is therefore not caused by zeros
entering from outside the Gram strip. What happens instead is precisely the
collision mechanism which confinement was designed to isolate.

To see this, define the Davenport--Heilbronn Gram discriminant by
\[
        \Delta_{44}^{DH}(r)
        =
        Z_N^{DH}
        \bigl(
        g_{44}^{DH}(r);r
        \bigr).
\]
In striking contrast with the isolated bad Gram point \( g_{126} \) of the
Hardy \( Z \)-function, the first-order approximation
\(
       Z_N^{DH}(g_{44}^{DH};r)
\)
tracks the corrected discriminant very closely. There is essentially no
large nonlinear correction which restores the correct sign. Instead,
\(\Delta_{44}^{DH}(r)\) itself reaches zero. In the numerical deformation
shown below this occurs at
\[
        r_*
        \approx
        0.72346.
\]
At this parameter value,
\[
        \Delta_{44}^{DH}(r_*)=0,
\]
and equivalently
\[
        t_-^{DH}(r_*)
        =
        t_+^{DH}(r_*)
        =
        g_{44}^{DH}(r_*).
\]
The two real zero branches therefore collide at a multiple zero. Beyond the
collision they cease to form two distinct real zeros and continue as the
non-real pair responsible for the failure of the Davenport--Heilbronn
analogue of the Riemann Hypothesis.

Figure~\ref{fig:DH_collision} makes the contrast with the Hardy
\( Z \)-function particularly transparent. For the isolated bad Gram point
\( g_{126} \), the incorrect first-order motion is compensated by a pronounced
lateral displacement of the extremum, and the corrected discriminant remains
of the proper sign. In the Davenport--Heilbronn example, by contrast,
\( g_{44}^{DH}(r) \) moves only weakly. The transverse response is therefore
too small to compensate for the motion toward the discriminant, and the two
zero branches meet.

\newpage

\begin{figure}[htbp]
        \centering
        \includegraphics[width=0.64\textwidth]{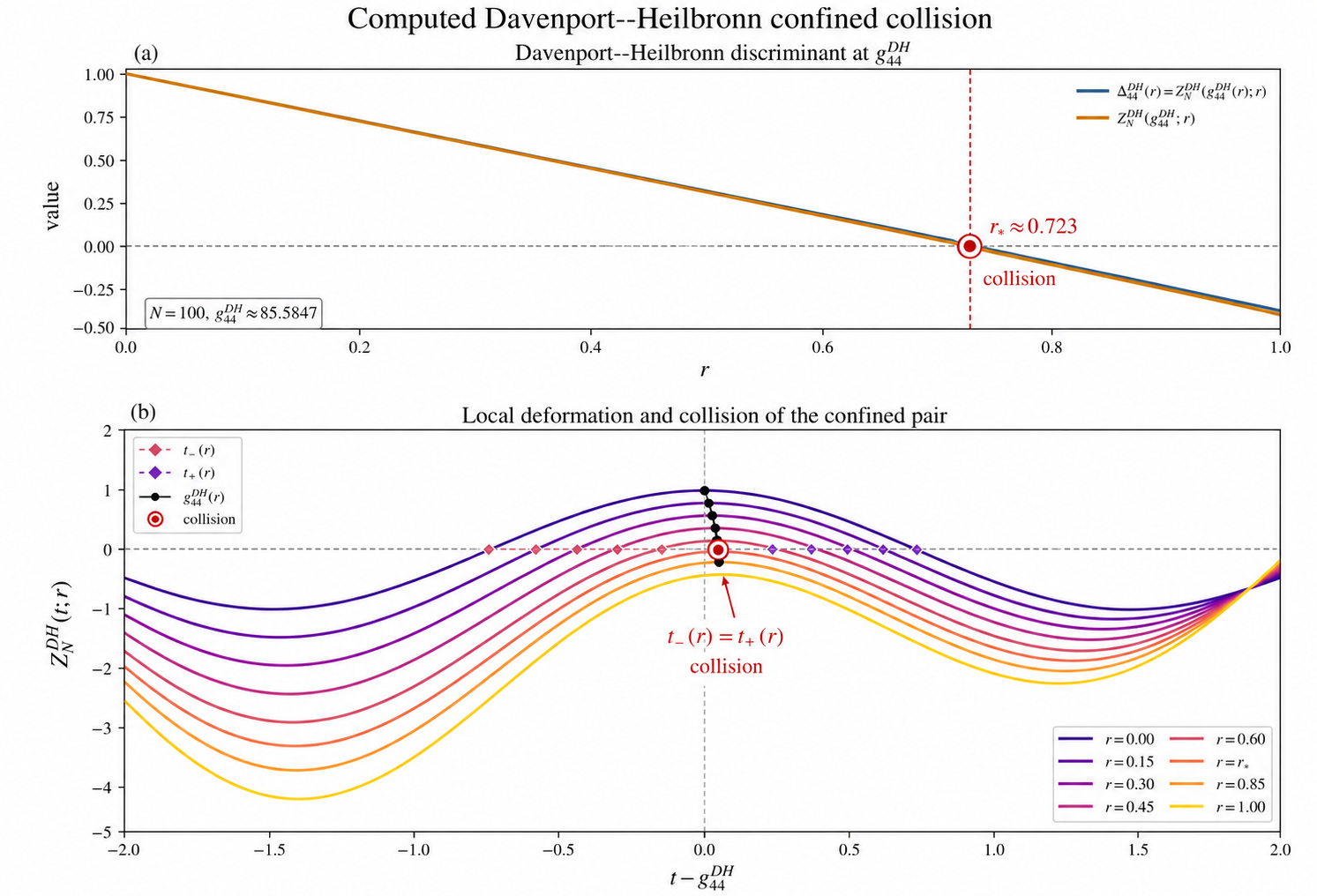}
        \caption{Local deformation of the confined zero pair associated
        with \( g_{44}^{DH} \) for the Davenport--Heilbronn function.
        The upper panel compares the corrected discriminant
        \( \Delta_{44}^{DH}(r) \) with the fixed-point first-order
        approximation \( Z_N^{DH}(g_{44}^{DH};r) \). The two remain closely
        aligned and cross the zero level at the collision parameter
        \( r_*\approx0.72346 \). The lower panel shows the corresponding
        family \( Z_N^{DH}(t;r) \), together with the continued extremum
        \( g_{44}^{DH}(r) \) and the real zero branches
        \( t_-^{DH}(r) \) and \( t_+^{DH}(r) \) up to their collision.
        Unlike the Hardy \( Z \)-function example, the extremum undergoes
        only a small lateral displacement and the confined pair actually
        collides.}
        \label{fig:DH_collision}
\end{figure} 

The corresponding animations, showing the dynamical version of  Fig. ~\ref{fig:DH_collision}, are available in the
accompanying Zenodo repository~\cite{JerbyZenodoGramPoints}.  They show the
motion of the confined zero pair and the continued extremum throughout the
deformation, including the collision and the subsequent departure of the two
zeros into the complex plane.

It should be noted that the analogous viscosity in this
Davenport--Heilbronn example is small: for the section used in
Figure~\ref{fig:DH_collision} one obtains numerically
\[
        \mu_{DH}(g_{44}^{DH})
        \approx 1.45,
\]
far below the conjectural threshold \(4\). Correspondingly, the transverse
motion is insufficient and an actual collision occurs. The Davenport--Heilbronn example therefore gives a concrete meaning to the
repulsion conjecture. The proposed inequality is not merely a quantitative
property of already real zeros; it distinguishes a collision-avoiding
dynamics from a dynamics in which a confined pair is allowed to leave the
real axis. In this sense, the repulsion relation is naturally interpreted as
a local mechanism protecting the real-zero configuration of the Hardy
\( Z \)-function, and hence as a phenomenon intimately connected with the
Riemann Hypothesis.

\section{The Full Gram Cloud $\mathcal{C}_N$ and its Multiplicative Layers}
\label{sec:clouds}

\subsection{Multiplicative-depth layers of the \(E(g_n)\)-cloud}
\label{subsec:afe_layers}

The complex cloud $\mathcal{C}_N$ given by values of \(E(g_n)\), 
contains contributions from all integers \(k\le N(g_n)\).  To examine how
the arithmetic structure of the summation index is reflected in the geometry
of the cloud, we decompose \(E(g_n)\) according to the total number of prime
factors of \(k\). Let
\[
        \Omega(k)
        :=
        \sum_{p^\alpha\Vert k}\alpha
\]
denote the number of prime factors of \(k\), counted with multiplicity.  For
\(j\ge 1\), define the \(j\)-th multiplicative-depth layer by
\begin{equation}
\label{eq:Ej_layer}
        E_j(g_n)
        :=
        \sum_{\substack{k\le N(g_n)\\ \Omega(k)=j}}
        \frac{e^{ig_n\log k}}{\sqrt{k}} .
\end{equation}
Thus \(E_j(g_n)\) is the part of the AFE cloud coming from integers with
exactly \(j\) prime factors.  In particular, the first layer is the
prime-indexed component
\begin{equation}
\label{eq:prime_layer}
        E_1(g_n)
        =
        \sum_{\substack{p\le N(g_n)\\ p\ {\rm prime}}}
        \frac{e^{ig_n\log p}}{\sqrt p}.
\end{equation}

\begin{proposition}[Multiplicative layer decomposition]
\label{prop:layer_decomposition}
For every Gram point \(g_n\), the AFE sum admits the finite decomposition
\begin{equation}
\label{eq:E_layer_decomposition}
        E(g_n)
        =
        1+\sum_{j\ge 1}E_j(g_n).
\end{equation}
Equivalently, the full cloud \(E(g_n)\) is obtained by adding the
multiplicative-depth layers, together with the contribution of the term
\(k=1\).
\end{proposition}

\begin{proof}
By definition,
\[
        E(g_n)
        =
        \sum_{k\le N(g_n)}
        \frac{e^{ig_n\log k}}{\sqrt{k}} .
\]
The term \(k=1\) contributes \(1\).  Every integer \(k\ge 2\) has a unique
prime factorization, and hence a well-defined value of \(\Omega(k)\).  The
sets
\[
        \{k\le N(g_n):\Omega(k)=j\},
        \qquad j\ge 1,
\]
therefore form a disjoint partition of the integers \(2\le k\le N(g_n)\).
Splitting the sum according to this partition gives
\[
        E(g_n)
        =
        1+
        \sum_{j\ge 1}
        \sum_{\substack{k\le N(g_n)\\ \Omega(k)=j}}
        \frac{e^{ig_n\log k}}{\sqrt{k}}
        =
        1+\sum_{j\ge1}E_j(g_n),
\]
as claimed.
\end{proof}

This decomposition gives a natural arithmetic filtration of the AFE phase
cloud.  The layer \(E_1(g_n)\) contains the prime frequencies, \(E_2(g_n)\)
contains the semiprime frequencies, and higher layers correspond to
increasing multiplicative complexity.  We now ask which of these layers
retain information about the local zero geometry, as measured by the
viscosity
\(
        \mu(g_n)
        =
        \left|\frac{Z'(g_n)}{Z(g_n)}\right|
\). To quantify this relation, we use Spearman's rank correlation coefficient
\[
        \rho_S(X,Y)
        =
        \operatorname{Corr}
        \bigl(\operatorname{rank}X,\operatorname{rank}Y\bigr),
\]
which measures monotone association between two quantities
\cite{Spearman1904}.  For each layer \(E_j(g_n)\), \(j=1,\ldots,5\), we
first fold the cloud by
\(        (\Re E_j(g_n),\Im E_j(g_n))
        \longmapsto
        (\Re E_j(g_n),|\Im E_j(g_n)|)
\).
Let \(T_{j,n}\) denote the coordinate of the folded point along the first
principal axis of the corresponding folded \(E_j\)-cloud. Figure
\ref{fig:spearman_layers} shows the absolute Spearman rank correlation
between \(T_{j,n}\) and \(\mu(g_n)\).

\begin{figure}[H]
\centering
\includegraphics[width=0.45\textwidth]{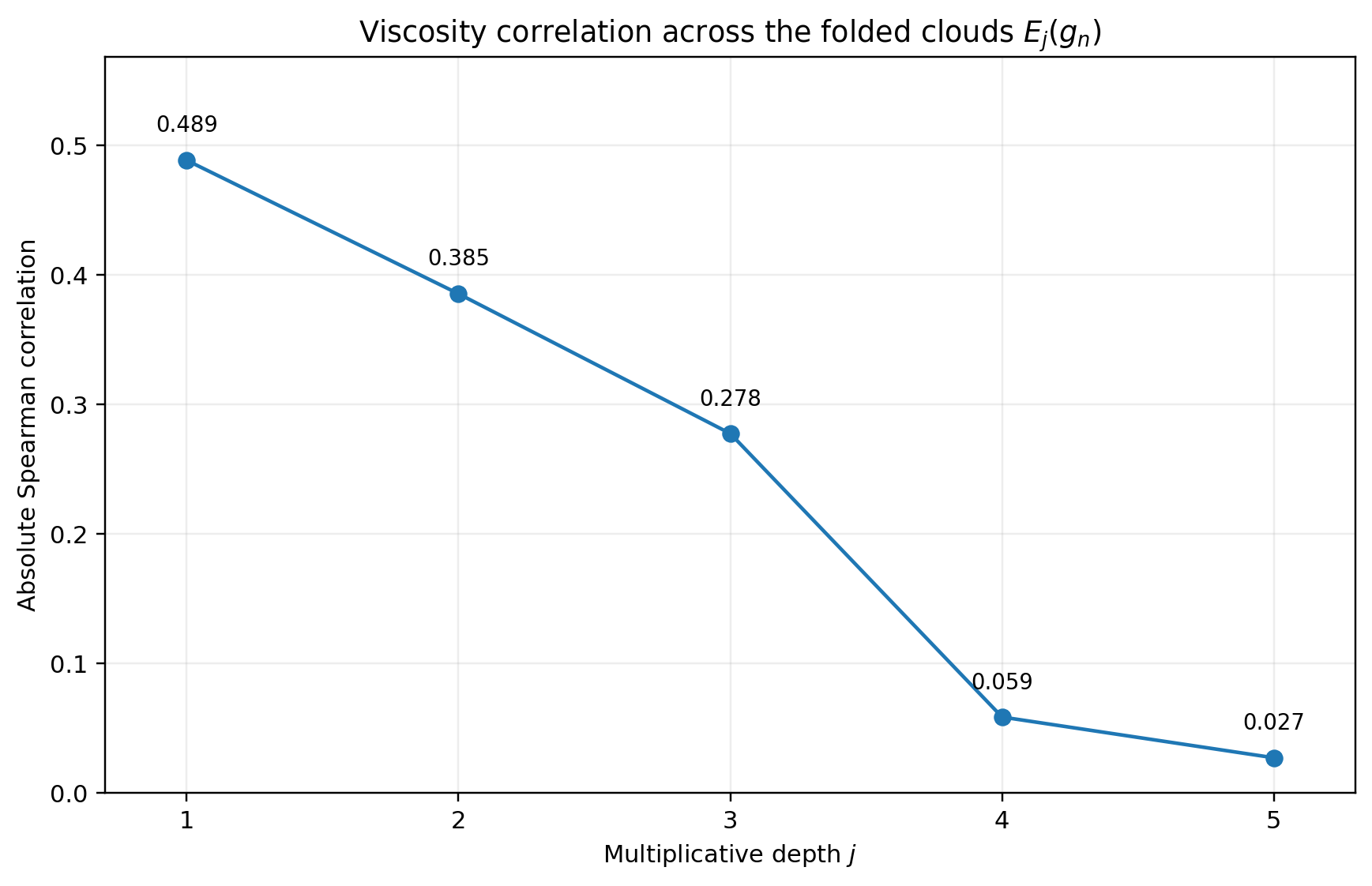}
\caption{
Absolute Spearman rank correlation between the viscosity
\(\mu(g_n)\) and the principal-axis coordinate \(T_{j,n}\) of the folded
\(E_j\)-cloud, for \(j=1,\ldots,5\).
}
\label{fig:spearman_layers}
\end{figure}

The figure shows a clear monotone decrease of the correlation as the
multiplicative depth \(j\) increases.  Thus, among the first few arithmetic
layers, the prime layer \(E_1(g_n)\) carries the strongest visible relation
to the viscosity. It should be kept in mind, however, that in the present data set
\(N(g_n)\le 164\).  Consequently, the higher-\(\Omega\) layers contain only
a limited number of frequencies:
\[
        \#\{k\le 164:\Omega(k)=j\}
        =
        38,54,37,21,8,
        \qquad j=1,\ldots,5.
\]
For this reason the Spearman coefficient alone should not be regarded as a
complete description of the arithmetic structure.  It is a useful first
summary, but the two-dimensional geometry of the clouds contains additional
information.  We therefore return to the prime cloud $\mathcal{P}_N$ of values of \(E_1(g_n)\), where
the separation between good and isolated bad Gram points is most visible,
and then compare it with the higher multiplicative-depth layers. Figure~\ref{fig:prime_cloud} re-displays the prime $\mathcal{P}_N$
for the first \(N=250{,}000\) Gram points. This time with the \(40\) isolated bad Gram points with
the smallest values of \(\mu(g_n)\) are marked by yellow circles with black
boundaries.

\begin{figure}[H]
        \centering
        \includegraphics[width=0.45 \textwidth]{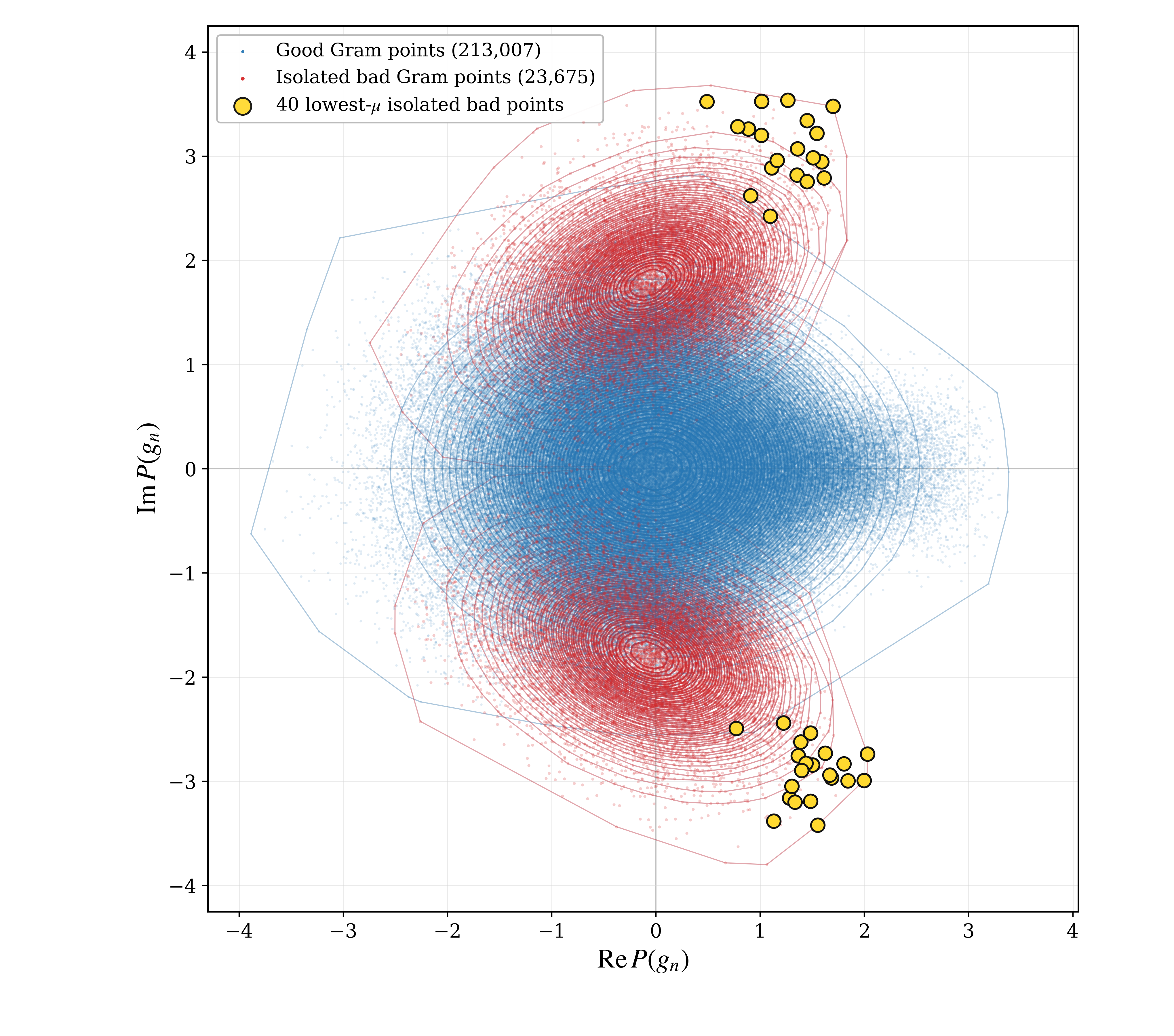}
        \caption{
The prime layer \(E_1(g_n)\) for the first \(250{,}000\) Gram points.
Good Gram points are shown in blue and isolated bad Gram points
in red. The \(40\) isolated bad Gram points of
smallest \(\mu(g_n)\) are marked by yellow circles.
The nested onion curves are obtained by successive convex-hull peeling.}
        \label{fig:prime_cloud}
\end{figure}

To make the geometry of the two point clouds more visible, we superimpose
their successive convex hull peeling layers to which we will refer as \emph{onion layers}. For the good points, these are obtained by taking
the convex hull of the blue cloud, removing the points on its boundary, and
then repeating the procedure on the remaining points. For the isolated bad
points, the same convex-hull peeling is carried out separately on the upper
and lower red components. 
The two classes display markedly different geometries. The good Gram
points form a broad cloud concentrated near the real direction,
whereas the isolated bad points occupy two pronounced lobes above
and below the real axis, and the
        lowest-viscosity configurations are concentrated near
        distinguished extremal regions of these lobes. 

In fact, the viscosity repulsion phenomenon is already visible in the prime cloud. To compare the two bad lobes of the prime cloud $\mathcal{P}_N$
on the same footing, we fold the cloud by
\[
        (x,y)\longmapsto (x,|y|).
\]
After this folding, the upper and lower bad lobes are superimposed into a
single folded prime cloud. The result is shown in Figure~\ref{fig:folded_prime_viscosity}.  The points
are colored according to \(\ln(\mu(g_n)-4)\), so that points close to the
conjectural repulsion boundary \(\mu=4\) are distinguished from points with
larger viscosity.

\begin{figure}[htbp]
        \centering
        \includegraphics[width=0.5\textwidth]{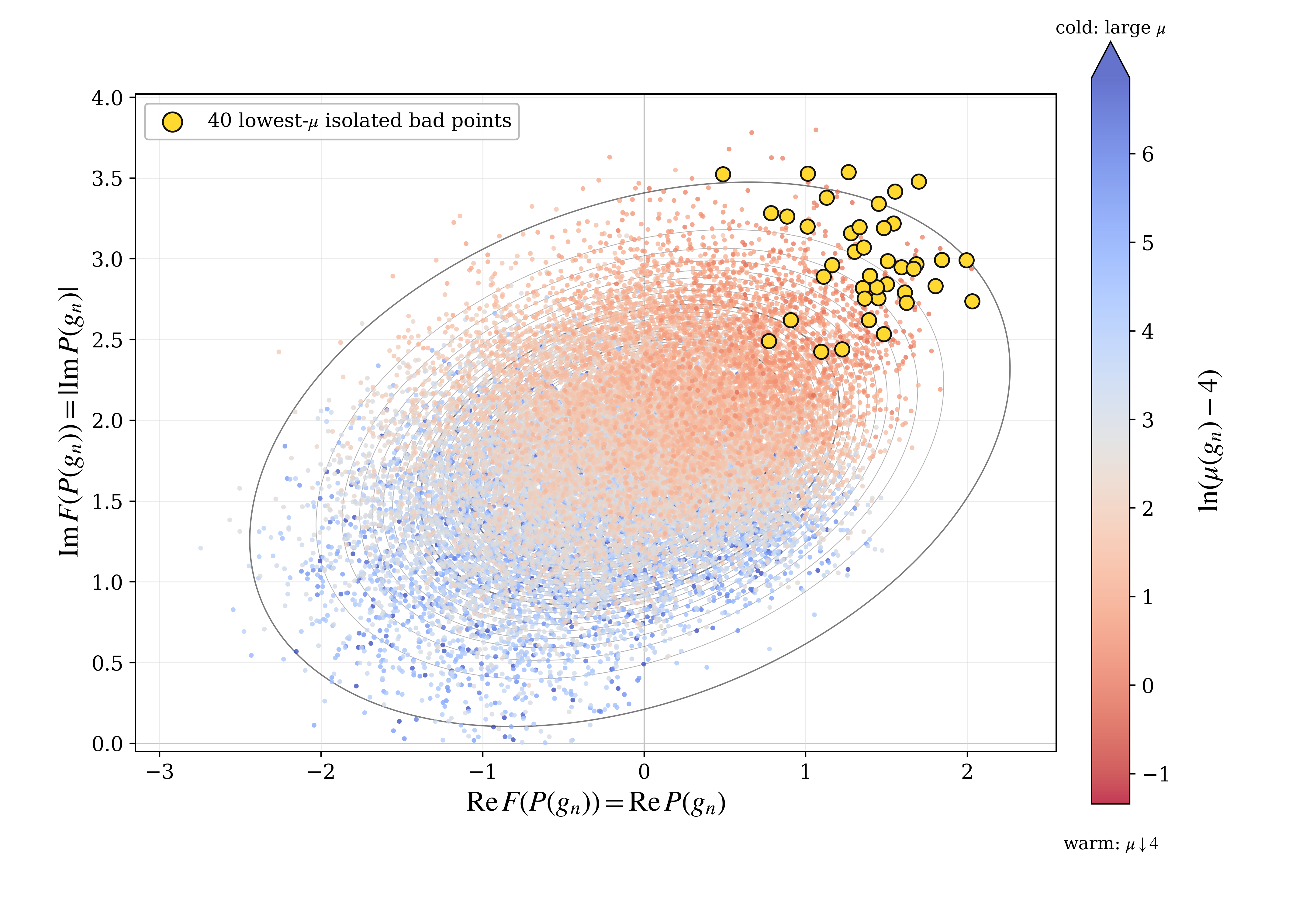}
        \caption{
        Folded prime cloud for the isolated bad Gram points, colored
        according to \(\ln(\mu(g_n)-4)\).  The \(40\) isolated bad Gram
        points with the smallest values of \(\mu(g_n)\) are highlighted in
        yellow.
        }
        \label{fig:folded_prime_viscosity}
\end{figure}

The main point is that the viscosity is not randomly distributed across the
folded cloud.  Rather, the folded bad cloud has a clear principal direction,
and \(\mu(g_n)\) varies systematically along this direction.  The smallest
viscosities, namely the configurations closest to the repulsion boundary
\(\mu=4\), are concentrated near one extremal end of the folded lobe.  As one
moves along the major axis of the cloud, the viscosity increases. Thus the prime cloud contains more information than the good--bad
distinction alone.  It does not merely detect that a Gram point is isolated
bad, it also organizes the isolated bad configurations according to the
strength of the local repulsion measured by
\(\mu(g_n)\).
This gives a geometric explanation of the Spearman correlation observed
above as the principal-axis coordinate of the folded prime cloud is seen to function as
a visible ordering parameter for the viscosity.

\subsubsection{The higher multiplicative layers $E_j(g_n)$ for $j \geq 2$}

We next compare the first five multiplicative-depth layers directly.
Figure~\ref{fig:Ej_montage} displays the clouds \(E_j(g_n)\),
\(1\leq j\leq5\), with good Gram points shown in blue and isolated bad
Gram points in red.

\begin{figure}[htbp]
        \centering
        \includegraphics[width=0.9 \textwidth]
        {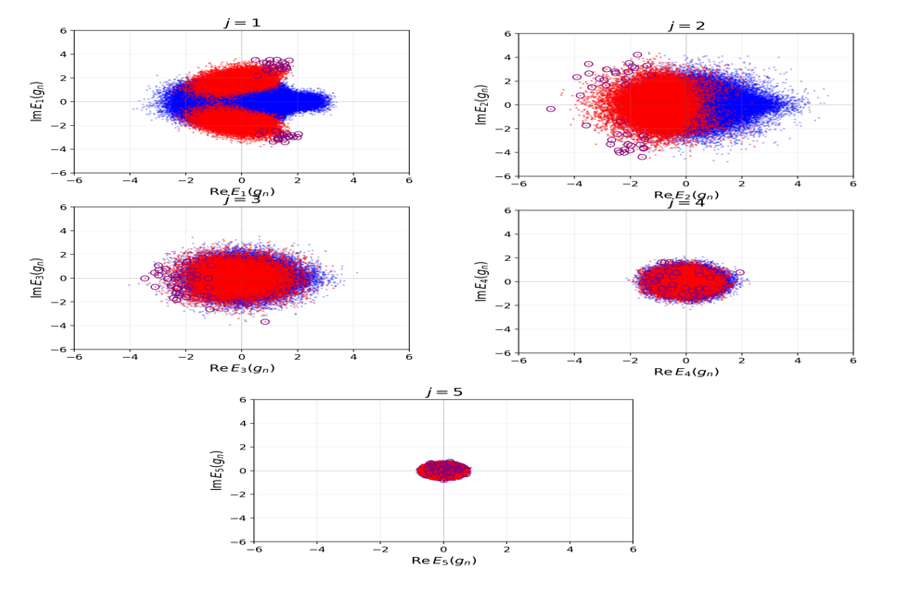}
        \caption{
        The first five multiplicative-depth layers
        \(E_j(g_n)\), \(1\leq j\leq5\), shown on a common scale.
        Good Gram points are shown in blue and isolated bad Gram
        points in red.
        }
        \label{fig:Ej_montage}
\end{figure}

The visual progression in Figure~\ref{fig:Ej_montage} agrees closely
with the correlation hierarchy observed in \ref{subsec:afe_layers}. For \(j=1\), the distinction
between good and isolated bad Gram points is pronounced. The
separation remains clearly visible in the second layer
\(\Omega(k)=2\). In the third layer the two populations remain
distinguishable, but the organization is considerably weaker. By
\(j=4\) and \(j=5\), the red and blue clouds are largely intermixed,
and the viscosity-dependent directional organization found in the
first layers has essentially disappeared.

\subsection{A probabilistic interpretation of the multiplicative layers}
\label{subsec:theory_layers}

The geometry observed in the first multiplicative layers admits a
natural interpretation in terms of random phases and large-deviation
theory. For a comprehensive treatment of large-deviation theory, we refer the reader
to~\cite{DemboZeitouni}. The prime layer is the most transparent case.

\subsubsection{The prime layer}

Heuristically, as the Gram point \(g_n\) varies, the prime phases
\[
        \Theta_p(g_n):=g_n\log p \pmod{2\pi}
\]
behave like weakly dependent angular variables.  This is the Gram-point
analogue of the standard random-phase heuristic for zeta Dirichlet
polynomials: when \(t\) varies over a long interval, the phases
\(t\log p\) for distinct primes are treated as approximately independent
rotations, reflecting the rational independence of the numbers \(\log p\)
and underlying the random Euler product model for \(\zeta(s)\);
see \cite{IvicZ,MontgomeryVaughanMNT,Titchmarsh}. The prime layer
\(
        E_1(g_n)
        =
        \sum_{p\leq N(g_n)}
        \frac{e^{ig_n\log p}}{\sqrt p}
\)
may therefore be viewed as a planar vector sum
\[
        P_N
        =
        \sum_{p\leq N}
        \frac{1}{\sqrt p}
        \bigl(
        \cos\Theta_p,\,
        \sin\Theta_p
        \bigr).
\]
Each prime contributes a vector of length \(p^{-1/2}\), pointing in the
direction determined by the phase \(\Theta_p\). The endpoint \(P_N\) is thus
the result of adding many small vectors with different directions. This is
the sense in which the prime cloud resembles a weighted planar random walk.

\begin{remark}[Random phases, primes, zeros, and bad Gram points] It should be noted that in the classical random-phase heuristic model, the
prime sum
\[
        \sum_{p\le X}\frac{e^{it\log p}}{\sqrt p}
\]
is typically considered to behave like a weighted planar random walk. In this work \(t\) is sampled only at the Gram points \(g_n\), which already shifts from the classical heuristic setting. Rather, it gives a natural ``null model'' for the prime cloud
\(\mathcal{P}_N\). 

From this point of view, the striking feature of the isolated bad Gram
points is that they do not behave like generic samples from this random
phase model. This suggests that the condition of being an isolated bad Gram point imposes
a nontrivial constraint on the prime phases. The reason this is plausible is the prime--zero duality expressed by the
explicit formula. In particular, prime sums and zero configurations are not independent
objects, but two sides of the same oscillatory structure.  Thus a special
local zero configuration near \(g_n\), such as the confined pair associated
with an isolated bad Gram point, may leave a visible signature in the prime
projection \(E_1(g_n)\).  In this interpretation, the prime lobes are not a
generic random-walk feature. Rather, they are the geometric trace of conditioning the
prime phases on a special local zero configuration.
\end{remark}

We now fix a class \(\mathcal C\) of Gram points. For example,
\(\mathcal C\) may denote all good Gram points, the upper isolated-bad lobe,
or the lower isolated-bad lobe. We then regard \(P_N\) as varying only over
the Gram points in this class. Its conditional mean is
\[
        m_{\mathcal C,N}
        :=
        \mathbb E_{\mathcal C} P_N,
\]
that is, the average position of the corresponding cloud, and its covariance
matrix is
\[
        \Sigma_{\mathcal C,N}
        :=
        \operatorname{Cov}_{\mathcal C}(P_N).
\]
The covariance matrix records the shape of the cloud near its center. If the
cloud spreads equally in every direction, then
\(\Sigma_{\mathcal C,N}\) is approximately a scalar multiple of the
identity matrix and the cloud is locally circular. If the cloud spreads
more strongly in one direction than another, then
\(\Sigma_{\mathcal C,N}\) has unequal eigenvalues and the cloud is locally
elliptical.

To describe not only the central bulk but also the rarer points farther away
from the center, it is convenient to introduce the logarithmic
moment-generating function
\[
        K_{\mathcal C,N}(u)
        :=
        \log
        \mathbb E_{\mathcal C}
        \exp\!\left(
        \langle u,P_N\rangle
        \right),
        \qquad
        u\in\mathbb R^2.
\]
Here
\[
        \langle u,P_N\rangle
        =
        u_1\operatorname{Re}P_N
        +
        u_2\operatorname{Im}P_N
\]
is the component of the random vector \(P_N\) in the direction \(u\).
The exponential factor
\(
        \exp\!\left(
        \langle u,P_N\rangle
        \right)
\)
gives extra weight to points lying far in the direction \(u\). Thus
\(K_{\mathcal C,N}(u)\) measures, in a compressed form, how strongly the
cloud extends in each possible direction.

Under the random-phase approximation, the contribution of each prime is
treated as approximately independent of the others. Since the exponential
of a sum factors into a product,
\[
        \exp\!\left(
        \left\langle
        u,
        \sum_{p\leq N}X_p
        \right\rangle
        \right)
        =
        \prod_{p\leq N}
        \exp\!\left(
        \langle u,X_p\rangle
        \right),
\]
where
\[
        X_p
        :=
        \frac{1}{\sqrt p}
        \bigl(
        \cos\Theta_p,\,
        \sin\Theta_p
        \bigr),
\]
independence would imply that the expectation of the product is the product
of the expectations. Taking logarithms then converts the product into a sum.
Thus one is led heuristically to
\[
        K_{\mathcal C,N}(u)
        \approx
        \sum_{p\leq N}
        \log
        \mathbb E_{\mathcal C}
        \exp\!\left(
        \frac{
        u_1\cos\Theta_p
        +
        u_2\sin\Theta_p
        }{\sqrt p}
        \right).
\]
This approximate additivity is one of the principal reasons why the prime
layer is especially tractable: the complicated geometry of the entire cloud
is expected to be built from a sum of individual prime contributions.

The next step is to associate to \(K_{\mathcal C,N}\) a function measuring
how unlikely it is for the cloud to reach a given point \(z\in\mathbb R^2\).
This is the role of the Legendre transform
\[
        I_{\mathcal C,N}(z)
        :=
        \sup_{u\in\mathbb R^2}
        \left\{
        \langle u,z\rangle
        -
        K_{\mathcal C,N}(u)
        \right\}.
\]
The function \(I_{\mathcal C,N}\) is called a rate function in
large-deviation theory. Informally, large-deviation theory says that points
with large values of \(I_{\mathcal C,N}(z)\) are statistically much less
likely to occur than points with small values. Thus \(I_{\mathcal C,N}\)
plays a role analogous to an energy landscape for the cloud: low values
describe the region where points are concentrated, while higher values
describe progressively rarer excursions away from the center.
This point of view is useful here because contours of equal value
\[
        I_{\mathcal C,N}(z)=\text{constant}
\]
should, heuristically, correspond to contours of comparable statistical
rarity. In other words, if two points lie on approximately the same level
set of \(I_{\mathcal C,N}\), one expects them to occur with roughly comparable
frequency.

The local shape of these contours can be understood using basic ideas from
large-deviation theory, see
~\cite{DemboZeitouni}. Near the mean
\(m_{\mathcal C,N}\), the logarithmic moment-generating function admits the
Taylor expansion
\[
        K_{\mathcal C,N}(u)
        =
        \langle u,m_{\mathcal C,N}\rangle
        +
        \frac12
        u^T\Sigma_{\mathcal C,N}u
        +
        O(\|u\|^3).
\]
The linear term contains the mean of the cloud, while the quadratic term
contains its covariance matrix. Keeping only these first two non-trivial
terms gives the Gaussian approximation to the cloud near its center.

For a quadratic function, the Legendre transform can be computed explicitly.
One obtains
\[
        I_{\mathcal C,N}(z)
        =
        \frac12
        (z-m_{\mathcal C,N})^T
        \Sigma_{\mathcal C,N}^{-1}
        (z-m_{\mathcal C,N})
        +
        O\!\left(
        \|z-m_{\mathcal C,N}\|^3
        \right).
\]
Thus, sufficiently near the center, the contours
\[
        I_{\mathcal C,N}(z)=c
\]
are approximately given by
\[
        (z-m_{\mathcal C,N})^T
        \Sigma_{\mathcal C,N}^{-1}
        (z-m_{\mathcal C,N})
        =
        2c.
\]
These are ellipses. Their principal directions are precisely the eigenvectors
of the covariance matrix
\(\Sigma_{\mathcal C,N}\), while the lengths of their principal semiaxes are
determined by the square roots of its eigenvalues. Thus the familiar
elliptical geometry seen in the prime cloud is exactly what one expects from
the second-order, or Gaussian, approximation to its local statistics.

This viewpoint also gives a heuristic interpretation of the empirical onion
layers obtained by repeated convex-hull peeling. A basic property of a
Legendre-transform rate function is convexity. Consequently, its sublevel
sets
\[
        \left\{
        z:
        I_{\mathcal C,N}(z)\leq r
        \right\}
\]
are nested convex sets. As \(r\) increases, these sets expand outward from
the most probable part of the cloud toward increasingly rare regions.

The empirical convex-hull peeling procedure of Fig. \ref{fig:prime_cloud} produces a similar nested
geometry directly from the finite data set. The outer convex hull is removed,
then the convex hull of the remaining points, and so on. We therefore regard
the onion layers as a finite-sample geometric analogue of the nested sublevel
sets of the rate function. Near the center, where the quadratic approximation
dominates, these layers should be approximately elliptical. Farther away,
systematic deviations from exact ellipses may reflect the cubic and higher
terms omitted from the local expansion of \(K_{\mathcal C,N}\), and hence
higher-order structure in the prime-phase distribution.

\begin{example}[Artificial random isotropic model] It is instructive to compare this with the simplest random-phase model.
Suppose that, for every prime \(p\), the phase \(\Theta_p\) were genuinely
uniform on \([0,2\pi)\), independently of all other prime phases. Then
\[
        \mathbb E(\cos\Theta_p)
        =
        \mathbb E(\sin\Theta_p)
        =
        0,
\]
and
\[
        \mathbb E(\cos^2\Theta_p)
        =
        \mathbb E(\sin^2\Theta_p)
        =
        \frac12,
        \qquad
        \mathbb E(\cos\Theta_p\sin\Theta_p)=0.
\]
Hence the covariance contributed by the \(p\)-th prime would be
\[
        \frac{1}{2p}I_2.
\]
Adding the prime contributions gives
\[
        \Sigma_N
        \approx
        \frac12
        \left(
        \sum_{p\leq N}\frac1p
        \right)
        I_2.
\]
The covariance would therefore be rotationally symmetric, and the local
contours would be circles.

\end{example} 

The prime-layer geometry, observed below, departs in a clear way from the simplest
isotropic random-phase picture. For good Gram points, the cloud is observed to have a
pronounced nearly circular central bulk, but with a visible deformation in
the positive-real direction. The isolated bad Gram points display a very
different geometry, separating into two distinct upper and lower lobes with
approximately elliptical shapes and preferred orientations. Thus the two
Gram-point classes are not obtained from one another by a simple change of
scale or by a uniform deformation of a common cloud. Rather, the conditioning
on the Gram-point type produces qualitatively different geometric structures
in the prime layer. 
\paragraph{Good Gram points: an isotropic core with a real-direction bias.}

Figure~\ref{fig:good_prime_cloud} displays the prime-layer geometry for the
good Gram points. The successive convex layers reveal a nearly circular
central bulk, together with a persistent bump in the positive-real
direction. Thus the dominant geometry is close to isotropic near the center,
while the outer layers exhibit a clear one-sided asymmetry.

\medskip
\begin{figure}[htbp]
        \centering
        \includegraphics[width=0.62\textwidth]{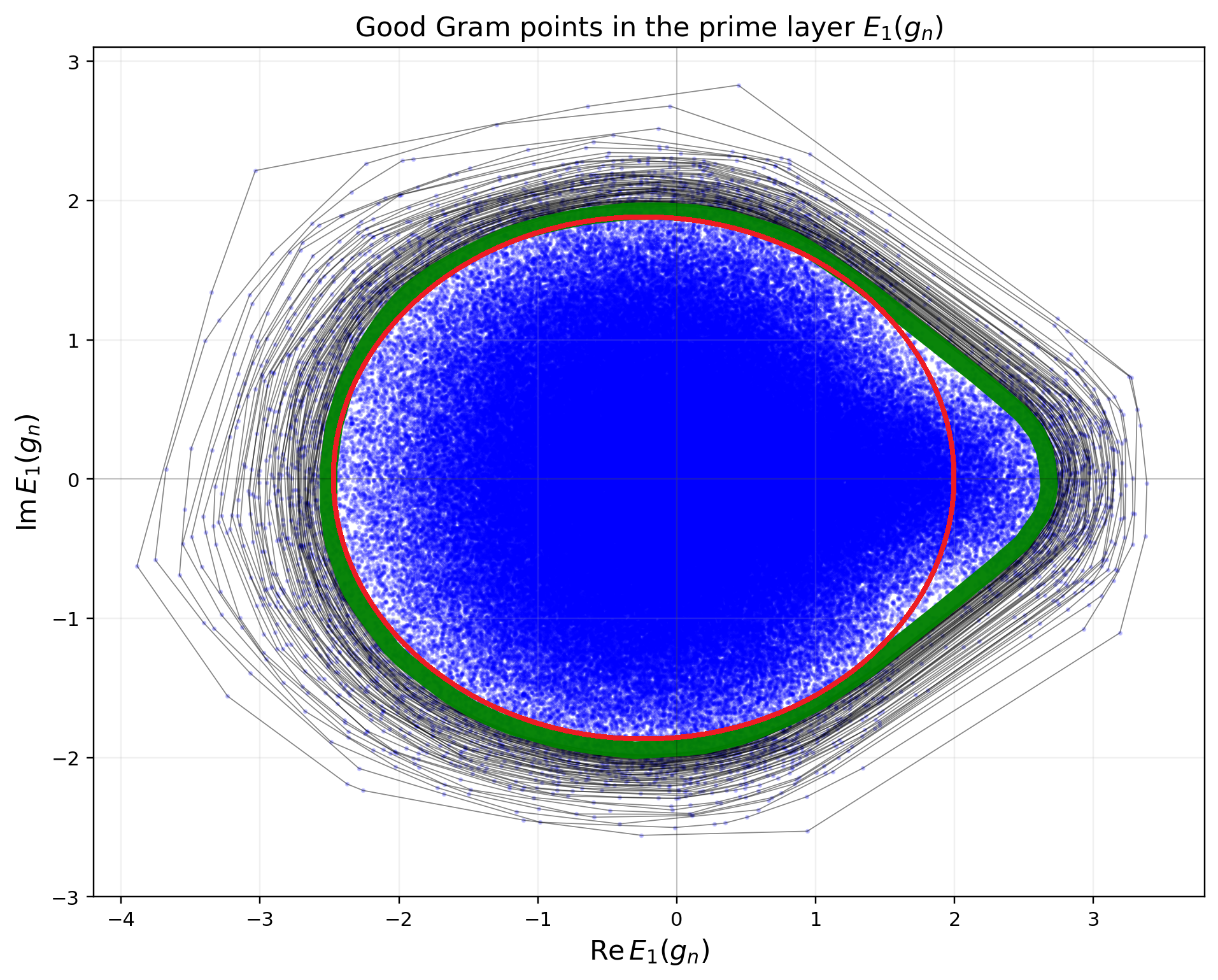}
        \caption{
        Prime-layer cloud \(E_1(g_n)\) for the good Gram points, together
        with its successive convex layers. The central region is nearly
        circular, while the outer layers display a persistent deformation
        in the positive-real direction.
        }
        \label{fig:good_prime_cloud}
\end{figure}
\medskip
It should be noted that there is a natural mechanism by which the Gram-sign
condition itself can produce the observed deformation in the positive-real
direction. The point is that goodness is defined by a sign condition on
\(Z(g_n)\), and at a Gram point this sign condition becomes, to leading order,
a condition on the \emph{real part} of the corresponding AFE sum. Indeed, at \(t=g_n\) we have
\[
        Z(g_n)
        =
        2(-1)^n\Re E(g_n)
        +
        O(g_n^{-1/4}),
\]
where
\[
        E(g_n)
        =
        1+E_1(g_n)+E_2(g_n)+\cdots.
\]
Since a good Gram point satisfies
\[
        (-1)^n Z(g_n)>0,
\]
the preceding approximation gives, schematically,
\[
        \Re E(g_n)>0.
\]
Thus the selection of good Gram points is not neutral with respect to the
complex AFE cloud. It preferentially selects those configurations for which
the total real component of the AFE sum is positive. Notice that this
constraint acts specifically in the real direction. There is no analogous
first-order sign restriction on the imaginary component. This already
suggests why the principal visible asymmetry of the good prime cloud should
occur horizontally rather than vertically.

To see how this condition can be transmitted to the prime layer, separate the
real part of the prime contribution from all remaining multiplicative
contributions. Write
\[
        X_n
        :=
        \Re E_1(g_n),
\]
and
\[
        R_n
        :=
        1+\sum_{j\geq2}\Re E_j(g_n).
\]
Then
\[
        \Re E(g_n)
        =
        X_n+R_n,
\]
and the good-Gram condition becomes, at the level of this approximation,
\[
        X_n+R_n>0.
\]
This formulation is useful because it makes clear that goodness does
\emph{not} require
\[
        X_n>0.
\]
A good Gram point may perfectly well have a negative prime contribution
\(X_n\), provided that the remaining contribution \(R_n\) is sufficiently
positive to compensate for it. Conversely, a large positive value of
\(X_n\) makes the inequality \(X_n+R_n>0\) easier to satisfy. Thus the
Gram-sign condition does not sharply truncate the prime cloud at
\(\Re E_1=0\). Rather, it changes the relative frequencies with which
different values of \(\Re E_1\) occur.

This can be expressed probabilistically. Suppose that \(X\) denotes the
prime-layer real coordinate and \(R\) the contribution of the remaining
layers. Before imposing the Gram-sign condition, let \(f_X(x)\) denote the
density of \(X\). After restricting to configurations satisfying
\[
        X+R>0,
\]
the density of \(X\) is reweighted according to
\[
        f_{X\,|\,X+R>0}(x)
        =
        \frac{
        f_X(x)\,
        \mathbb P(R>-x\,|\,X=x)
        }{
        \mathbb P(X+R>0)
        }.
\]
The important factor is therefore
\[
        \mathbb P(R>-x\,|\,X=x).
\]
For a fixed value \(X=x\), this is the probability that the remaining
multiplicative layers are large enough to make the total real part positive. The effect of this factor is if \(x\) is substantially
positive, then the requirement
\[
        R>-x
\]
is relatively mild and even moderately negative values of \(R\) are allowed.
If \(x\) is strongly negative, however, then \(R\) must be correspondingly
large and positive in order to rescue the Gram-sign condition. Such
configurations are naturally less common. Consequently, conditioning on
goodness tends to give greater statistical weight to configurations with
larger values of
\(
        X=\Re E_1(g_n)
\).
Even if the unconditioned prime cloud were approximately symmetric in the
real direction, the good-Gram selection can therefore produce an excess of
mass toward positive real values.
This gives a natural qualitative explanation for the rightward bump observed
in the good-point prime cloud. The nearly circular central bulk is consistent
with the underlying random-phase picture, in which the prime phases produce
an approximately isotropic planar sum. At the level of second-order statistics,
this geometry is described by the covariance matrix, which accounts for the
approximately circular spread of the cloud. While the rightward bump is of a different nature. The Gram-sign condition imposes
a one-sided constraint on the \emph{full} AFE sum, and this bias is already
partially visible in the prime component. Since the condition involves
\(X_n+R_n\), rather than \(X_n\) alone, it produces a smooth positive-real
deformation rather than a sharp cutoff. This asymmetry lies beyond the
second-order covariance description and should therefore be regarded as a
higher-order feature of the conditional prime-phase distribution.
\paragraph{Isolated bad Gram points: two transverse elliptic branches.}

The isolated-bad cloud presented in Fig.~\ref{fig:bad_prime_cloud}  exhibits a markedly different conditional
geometry. After separating the points into the upper and lower
half-planes, the onion layers of each lobe are strikingly close to
nested ellipses, with no comparable one-sided protrusion in the real
direction. Thus the quadratic large-deviation picture appears, at least
empirically, to describe each isolated-bad lobe considerably more
cleanly than it describes the entire good-point cloud.

\medskip

\begin{figure}[htbp]
        \centering
        \includegraphics[width=0.62\textwidth]{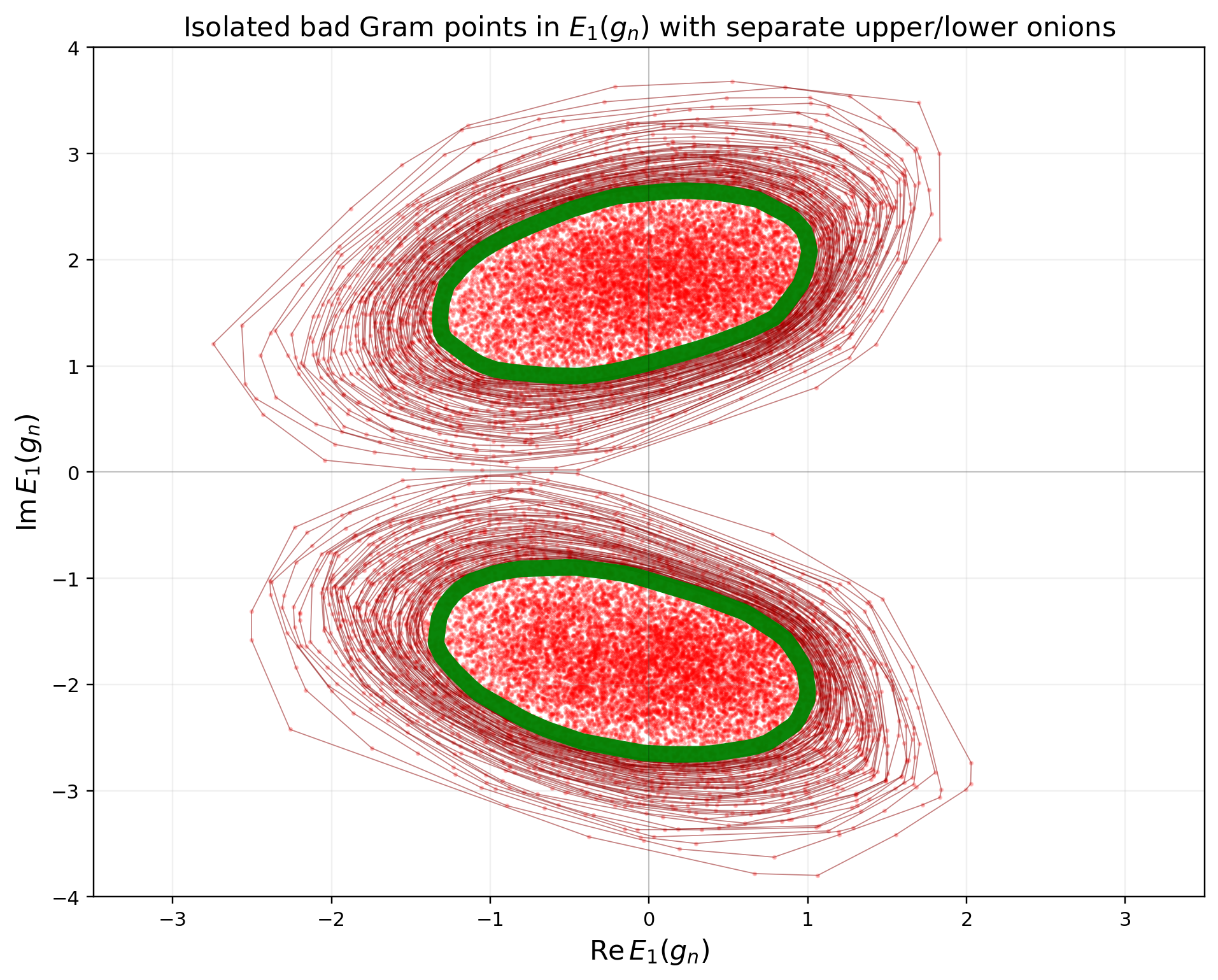}
        \caption{
        Prime-layer cloud \(E_1(g_n)\) for the isolated bad Gram points,
        together with the successive convex layers of the upper and lower
        lobes. Each lobe exhibits a strongly anisotropic, approximately
        elliptical geometry.
        }
        \label{fig:bad_prime_cloud}
\end{figure}

\medskip
The isolated-bad cloud appears to be governed much more strongly by its
second-order geometry. Thus, to a
first approximation, each lobe is well described by its mean and covariance
matrix. The displacement of the center determines the location of the lobe,
while the eigenvectors and eigenvalues of the covariance matrix determine its
preferred direction and degree of elongation.

This is notably different from the good-point cloud. There the covariance
captures the nearly circular bulk but not the additional rightward bump,
which is a genuinely higher-order asymmetry. For the isolated bad points,
by contrast, most of the visible structure is already present at the
second-order level. The principal phenomenon is therefore not a localized
deformation of an otherwise isotropic cloud, but the splitting into two
anisotropic branches, each of which is approximately elliptical.

\subsubsection{Higher multiplicative layers}

Figure~\ref{fig:Ej_montage} shows how the class-dependent geometry evolves
across the multiplicative layers \(E_j\). The strongest distinction between
good and isolated bad Gram points occurs in the prime layer \(E_1\). A related
structure remains visible in \(E_2\) and \(E_3\), but it weakens substantially
by \(E_4\) and \(E_5\). This also has a natural probabilistic interpretation. If
\[
        k=p_1\cdots p_j,
        \qquad \Omega(k)=j,
\]
then
\[
        e^{ig_n\log k}
        =
        \prod_{\nu=1}^{j} e^{ig_n\log p_\nu}.
\]
Hence
\[
        E_j(g_n)
        =
        \sum_{\substack{k\leq N(g_n)\\ \Omega(k)=j}}
        \frac{e^{ig_n\log k}}{\sqrt{k}}
\]
is built from degree-\(j\) products of the same primitive prime phases
\(e^{ig_n\log p}\). For \(j\geq2\), the summands therefore cease to behave
like new independent planar increments as different terms share prime factors,
and mixed moments and cumulants become increasingly important.

As the multiplicative depth grows, each phase combines more of the underlying
prime phases, and the resulting overlap can obscure the directional
information that is most clearly visible in \(E_1\). This provides a natural
mechanism for the progressive loss of class-dependent anisotropy seen in
Figure~\ref{fig:Ej_montage}. The large-scale geometry of the higher layers may
still be approximately described at second order by their conditional means
and covariance matrices, but the distinction between the good and isolated
bad classes becomes increasingly weak.

\section{The Bad Prime Lobes and Zeros Nearby to the Gram Point}
\label{sec:4}

Note that for an isolated bad Gram point \(g_n\), assuming the Riemann
Hypothesis, the interval
\(
[g_{n-1},g_{n+1}]
\)
contains two zeros of \(Z(t)\). Empirically, these two zeros lie on the
same side of \(g_n\) and either both lie in
\(
(g_n,g_{n+1}),
\)
or both lie in
\(
(g_{n-1},g_n).
\)
In fact, this dichotomy is directly responsible for the division into the two lobes of
the prime cloud. The upper lobe corresponds to one of these two local
zero configurations, while the lower lobe corresponds to the other.

Let
\(
        t_-<t_+
\)
denote the two zeros in \([g_{n-1},g_{n+1}]\). For a point in the upper
lobe, where both zeros lie to the right of \(g_n\), define
\[
        F_+(g_n)
        =
        \frac{
        \operatorname{dist}(t_-,g_{n+1})
        \operatorname{dist}(t_+,g_{n+1})
        }{
        \operatorname{dist}(t_-,g_n)
        \operatorname{dist}(t_+,g_n)
        }.
\]
Similarly, for a point in the lower lobe, where both zeros lie to the
left of \(g_n\), define
\[
        F_-(g_n)
        =
        \frac{
        \operatorname{dist}(t_-,g_{n-1})
        \operatorname{dist}(t_+,g_{n-1})
        }{
        \operatorname{dist}(t_-,g_n)
        \operatorname{dist}(t_+,g_n)
        }.
\]

The quantities \(F_\pm(g_n)\) are meant to give a simple measure of where the
two zeros sit between the bad Gram point and the neighboring good Gram point.
For the upper lobe, \(F_+(g_n)\) compares the pair's distance from
\(g_{n+1}\) with its distance from \(g_n\). For the lower lobe,
\(F_-(g_n)\) makes the analogous comparison with \(g_{n-1}\). Thus \(F_\pm(g_n)\) tells us whether the zero pair is positioned closer to the
central bad Gram point or closer to the outer neighboring Gram point. Small
values mean that the pair lies relatively close to the outer Gram point,
while large values mean that it lies relatively close to \(g_n\). The
quantity therefore serves as a simple scalar measure of the position of the
zero pair inside the corresponding half of the Gram interval. Figure~\ref{fig:extremal_lobes} displays the two isolated-bad prime
lobes
\[
\bigl(\Re E_1(g_n),\Im E_1(g_n)\bigr),
\]
colored according to \(F_+(g_n)\) on the upper lobe and
\(F_-(g_n)\) on the lower lobe. Warm colors correspond to small values
of \(F_\pm\), while cold colors correspond to large values. The forty
isolated bad Gram points having the smallest viscosities are marked
separately. The light curves indicate the approximate geometric
boundaries of the two lobes.

\begin{figure}[htbp]
    \centering
    \includegraphics[width=0.52\textwidth]{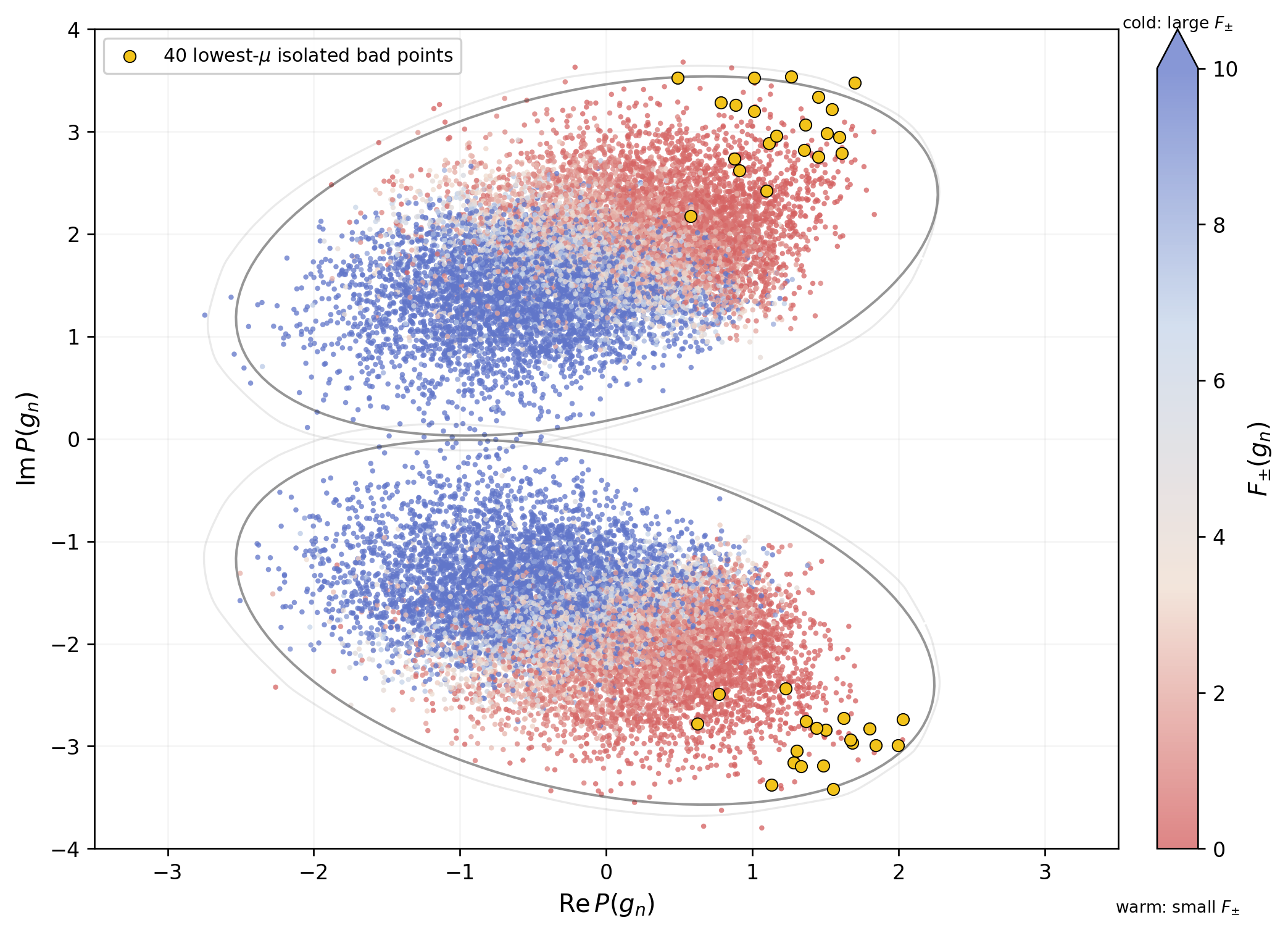}
    \caption{
    The two isolated-bad prime-cloud lobes colored according to the
    zero-position quantities \(F_\pm(g_n)\). The upper and lower lobes
    correspond to the two possible sides of \(g_n\) on which the local
    pair of zeros occurs. The \(40\) isolated bad Gram points with the
    smallest values of \(\mu(g_n)\) are marked by yellow circles.
    }
    \label{fig:extremal_lobes}
\end{figure}

Figure~\ref{fig:extremal_lobes} shows that the extremal configurations have a
particularly clear geometric interpretation. As one moves toward the
right-hand tip of either lobe, the corresponding zero pair is displaced
progressively toward the outer boundary of the local Gram interval. In the
upper lobe, this means that the rightmost zero \(t_+\) is pushed increasingly
close to \(g_{n+1}\). In the lower lobe, the leftmost zero \(t_-\) is pushed
increasingly close to \(g_{n-1}\).

The points with the smallest viscosities are concentrated precisely in these
extremal regions. Thus the regime in which \(\mu(g_n)\) approaches its
conjectural lower bound is not associated with an arbitrary rare prime-phase
configuration. Rather, it corresponds to a very specific zero geometry in
which one of the two zeros is pushed as far as possible toward the neighboring
good Gram point, that is, toward one of the walls \(g_{n\pm1}\) of the local
confined configuration.

To illustrate this extremal zero geometry directly, Fig.~\ref{fig:small_visc}
shows the graphs of \(\ln|Z(t)|\) across the local Gram interval
\([g_{n-1},g_{n+1}]\) for the two representative lowest-viscosity configurations $n=8569,96237$.
The two confined zeros are marked in each panel.  In these examples the pair
is visibly displaced toward the $g_{n-1}$ Gram wall, making concrete
the geometric interpretation suggested by the prime-cloud extremal regions.
\begin{figure}[H]
    \centering
    \includegraphics[width=0.62\textwidth]
{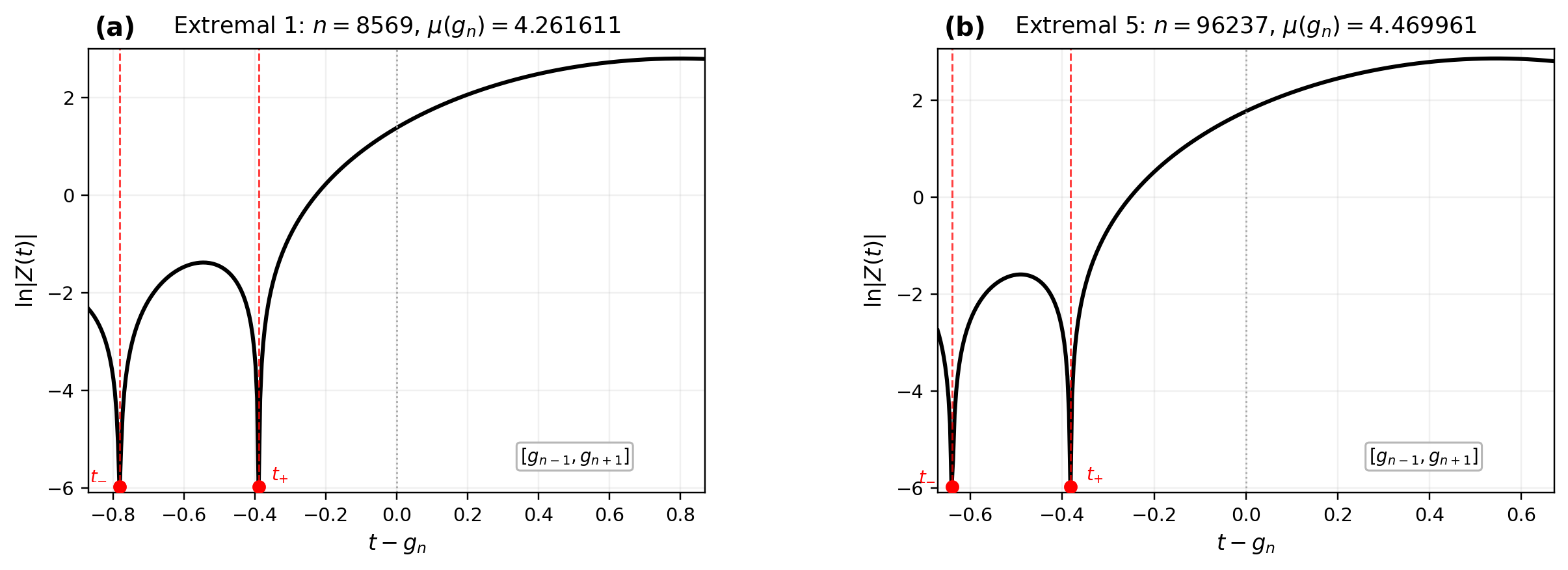}
    \caption{
    The local graphs  \(\ln|Z(t)|\) across the local Gram interval
\([g_{n-1},g_{n+1}]\) for the two representative lowest-viscosity configurations $n=8569,96237$.
    The two confined zeros are marked in red.  Their displacement toward the $g_{n-1}$ Gram point illustrates the characteristic zero geometry
    associated with viscosities close to the conjectural lower bound.
    }
    \label{fig:small_visc}
\end{figure}
\begin{remark}[Partial-sum trunks near extremal configurations] For \(t\in\mathbb R\), define the partial AFE sums
\[
        S_j(t)
        :=
        \sum_{k=1}^{j}
        \frac{e^{it\log k}}{\sqrt{k}},
        \qquad
        1\le j\le N(t).
\]
Thus \(S_{N(t)}(t)=E(t)\).  The curve
\(
        \{S_j(t):1\le j\le N(t)\}
\)
will be called the \emph{partial-sum trunk} at \(t\). We use these trunks only as a diagnostic.  For an isolated bad Gram point
\(g_n\), we compare the three trunks at
\(
        g_{n-1}, g_n, g_{n+1}
\).
Since
\(
        g_{n+1}-g_n
        \sim
        \frac{\pi}{\theta'(g_n)} \) and
       \(
        \theta'(g_n)
        \sim
        \frac12\log\!\left(\frac{g_n}{2\pi}\right)
\),
one Gram step changes the \(k\)-th phase by
\(
        (g_{n+1}-g_n)\log k
        \sim
        \pi\frac{\log k}{\log N}
\).
Thus low-index terms rotate very little from one Gram point to the next,
whereas terms near \(k=N\) rotate by almost \(\pi\).  This simple phase shear
is useful for visualizing how cancellation at the central bad point may be
converted into reinforcement at one of the neighboring good Gram points. Figure~\ref{fig:generic_trunks} shows representative isolated bad Gram points
away from the lowest-viscosity regime.  
\begin{figure}[H]
    \centering
    \includegraphics[width=0.75 \textwidth]{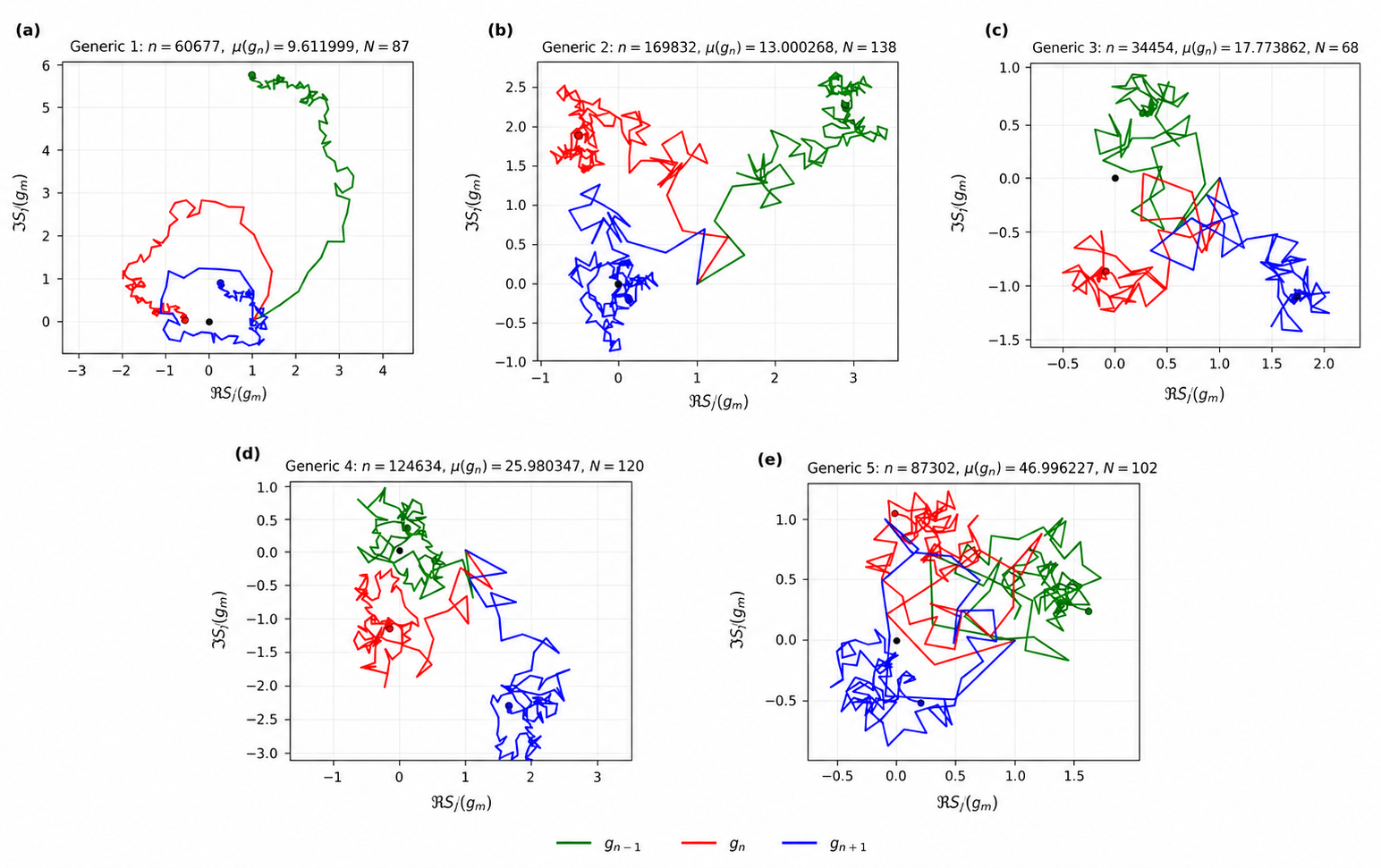}
    \caption{Trunks for five representative non-extremal isolated
    bad Gram points \(n=60677,169832,34454,124634,87302\).  For each triple,
    the trajectories \(S_j(g_{n-1})\), \(S_j(g_n)\), and \(S_j(g_{n+1})\) are
    shown in green, red, and blue, respectively.  Filled circles mark the
    terminal values \(S_N(g_m)\).}
    \label{fig:generic_trunks}
\end{figure}

No uniform three-trunk pattern is
visible in these generic examples.
A more coherent pattern appears among some of the isolated bad Gram points
with very small viscosity.  Figure~\ref{fig:low_visc_trunks} shows the five
smallest-viscosity examples in the present data set.

\begin{figure}[htbp]
    \centering
    \includegraphics[width=0.75\textwidth]{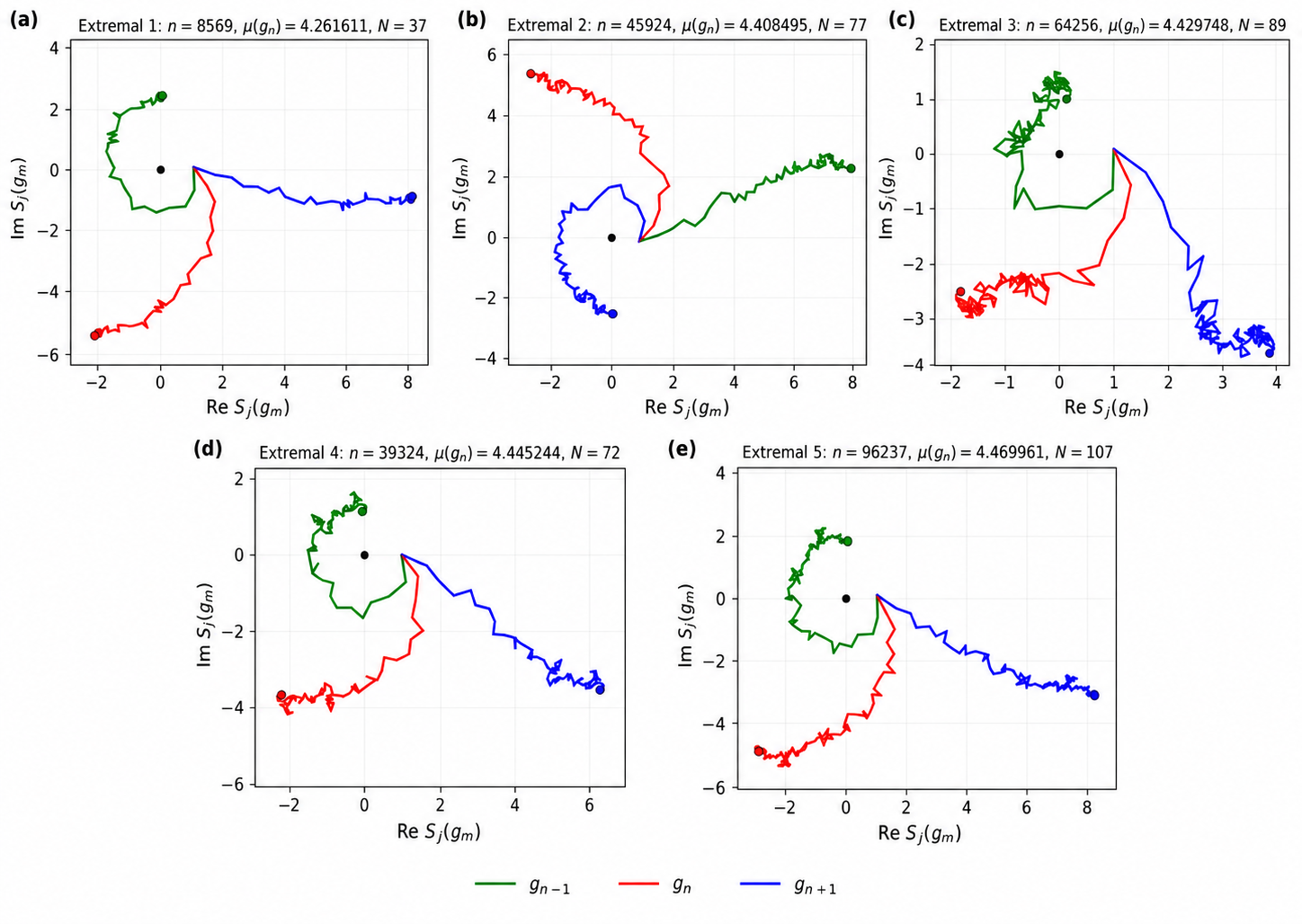}
    \caption{Trunks for the five lowest-viscosity isolated bad Gram
    points \(n=8569,45924,64256,39324,96237\).  The colors are as in
    Figure~\ref{fig:generic_trunks}.}
    \label{fig:low_visc_trunks}
\end{figure}

In these examples, one neighboring trunk terminates near a comparatively
small endpoint, while the opposite neighboring trunk reaches a much larger
endpoint.  Writing
\[
        A_{\mathrm{act}}
        :=
        \max\{|Z(g_{n-1})|,|Z(g_{n+1})|\},
\]
the observed coherent configurations suggest the rough slope estimate
\[
        |Z'(g_n)|
        \asymp
        \frac{A_{\mathrm{act}}}{g_{n+1}-g_n}
        \sim
        \frac{\theta'(g_n)}{\pi}A_{\mathrm{act}}.
\]
Consequently,
\[
        \mu(g_n)
        =
        \left|\frac{Z'(g_n)}{Z(g_n)}\right|
        \asymp
        \frac{\theta'(g_n)}{\pi}
        \frac{A_{\mathrm{act}}}{|Z(g_n)|}.
\]
This gives an intuitive explanation for why very small viscosity is difficult
to achieve and how the bound arises in such coherent configurations. The same phase organization that
produces cancellation at the central bad point tends, after one Gram step, to
produce a compensating neighboring excursion, and this excursion is measured
on the derivative scale by division by the Gram spacing.

We emphasize that this is only a qualitative mechanism.  We do not expect
every extremal point to arise from such a coherent three-trunk
configuration.  The figures of such coherent configurations should therefore be read as evidence for one
visible source of the viscosity barrier, not as a classification of all
near-extremal isolated bad Gram points.
\end{remark}

\subsection{The distribution of $\mu(g_n^{DH})$ for isolated bad DH Gram points}

Recall that Fig.~\ref{fig:viscosity_distributions} revealed a striking
separation between the viscosity distributions of good and isolated bad Gram
points for $Z(t)$.  While the good-point distribution is broad and extends
across a wide range of values, the isolated-bad distribution develops a sharp
lower edge immediately above
\(
\mu=4.
\)
In the data considered there, no isolated bad Gram point associated with a
confined pair was observed below this threshold, motivating the repulsion Conjecture.

We now perform the analogous comparison for the Davenport--Heilbronn
function. Fig.~\ref{fig:dh_visc_dist} shows the
corresponding empirical distributions of
$\ln\mu_{DH}(g_n^{DH})$ over our dataset.

\begin{figure}[H]
\centering
\includegraphics[width=0.58\textwidth]{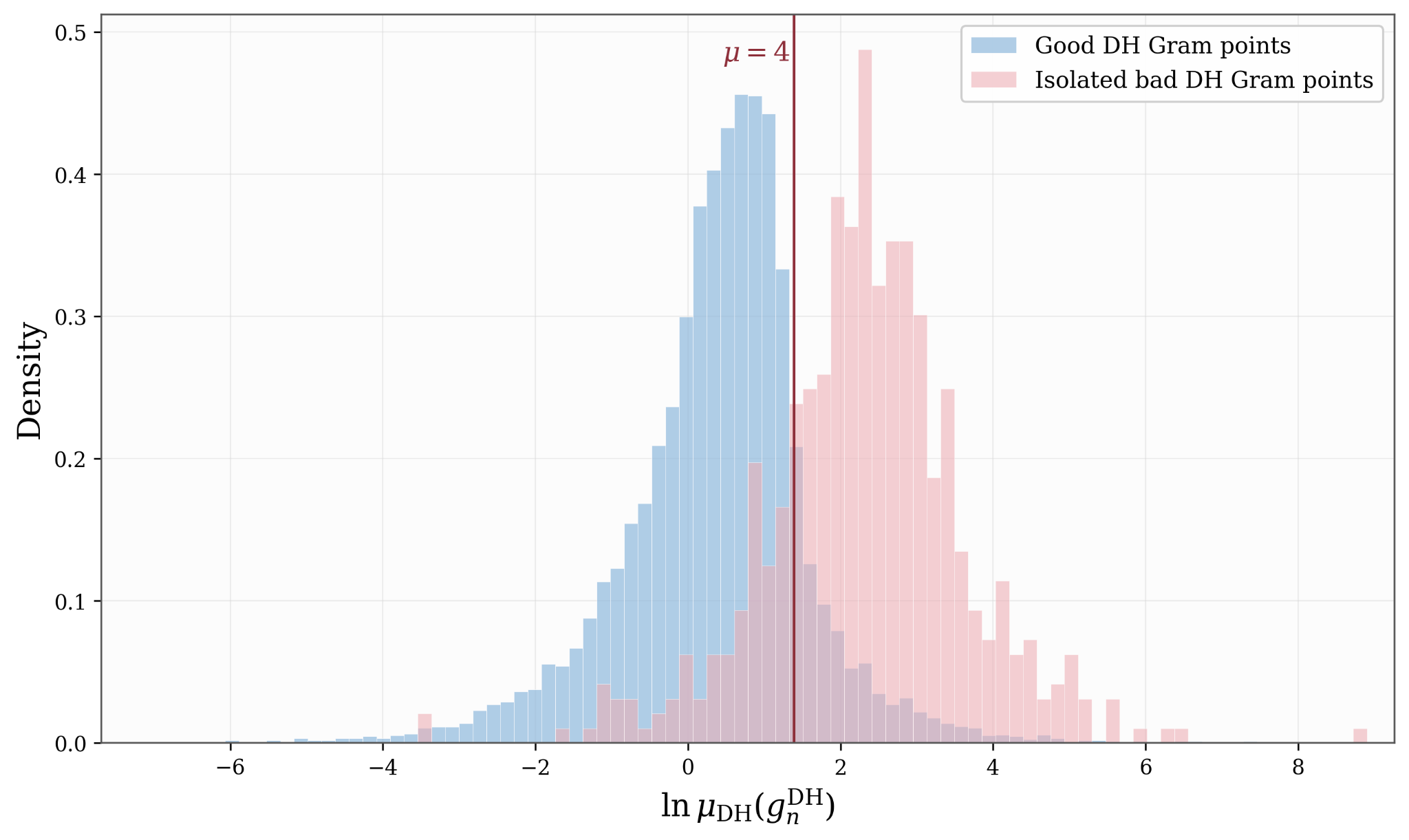}
\caption{Empirical distributions of
$\ln\mu_{DH}(g_n^{DH})$ for the $11\,525$ good DH Gram points (blue) and
$751$ isolated bad DH Gram points (red) in the data set
$2\leq n\leq12501$.  The vertical line marks $\mu=4$.  In sharp contrast
with Fig.~\ref{fig:viscosity_distributions}, the isolated-bad DH distribution
extends substantially to the left of this value and exhibits no lower
boundary at $\mu=4$.}
\label{fig:dh_visc_dist}
\end{figure}

The contrast with the Hardy $Z$-function is immediate.  For $Z(t)$, the
isolated-bad population is excluded from the region $\mu\leq4$ and is
strongly depleted near the boundary itself.  For the DH function, by
contrast, the isolated-bad distribution has substantial mass below
$\mu=4$ and overlaps broadly with the distribution of the good points.
Thus there is no visible analogue of the viscosity bound in the DH data.

This shows that the example discussed in Subsection~\ref{ss:2.3} is not an
isolated exception.  The failure of the viscosity relation is a statistical
feature of the DH function.  In particular, an isolated bad DH configuration
associated with a confined pair need not develop the strong transverse
response observed for $Z(t)$ and its viscosity may remain small enough for the
pair to reach the discriminant, collide, and leave the real axis.  The
difference between the two functions therefore appears already at the level
of their local Gram-point dynamics, well before one invokes the global fact
that the Davenport--Heilbronn function violates the analogue of the Riemann
Hypothesis.
\subsection{The $DH$-prime clouds $E_1(g_n^{DH})$}

In the $Z$-case, Fig.~\ref{fig:prime_cloud}
showed a highly structured picture with the isolated bad Gram points separated into two
distinct lobes, corresponding to the upper and lower parts of the folded
prime lung.  In particular, the isolated-bad set itself exhibited a clear
two-component organization.

For the Davenport--Heilbronn function, the situation is different.
Figure~\ref{fig:dh_dense_onions} shows the DH-weighted prime projection of
the same type of data.  The isolated bad DH Gram points, unlike in
the $Z$-case, do \emph{not} split into two separated prime-cloud
components.  Rather, they form a single coherent cloud.  For visual
convenience, we color the points with $\Im P_n^{DH}>0$ in orange and those
with $\Im P_n^{DH}<0$ in green, but this coloring should not be interpreted
as indicating two distinct geometric lobes analogous to the two lobes in
Fig.~\ref{fig:prime_cloud}.  The DH isolated-bad set is better regarded as
one cloud crossing the line $\Im P_n^{DH}=0$.

\begin{figure}[H]
\centering
\includegraphics[width=0.48\textwidth]{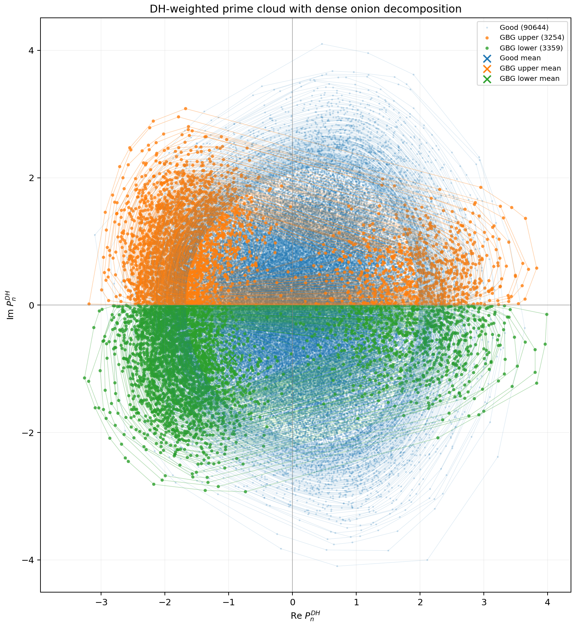}
\caption{The DH-weighted prime cloud with dense onion decomposition.
Blue points are good DH Gram points.  The isolated bad DH Gram points are
shown in orange and green according to the sign of $\Im P_n^{DH}$, but this
coloring is only for visual convenience. Unlike the isolated-bad cloud in
Fig.~\ref{fig:prime_cloud}, these points do not split into two separated
lobes, and instead form a single coherent cloud crossing
$\Im P_n^{DH}=0$.}
\label{fig:dh_dense_onions}
\end{figure}

Thus, from the perspective developed here, the DH function is seen to differ from the
Hardy $Z$-function in two complementary ways.  First, its local zero dynamics
do not exhibit the viscosity-bound phenomenon.  Second, the geometry of its
isolated-bad prime cloud differs qualitatively from the two-lung structure
observed for $Z(t)$.

\section{Theoretical Mechanisms}
\label{sec:mechanisms}
In this paper we have observed, empirically, a close relation between three
features of confined bad Gram points: the viscosity \(\mu(g_n)\), the local
configuration of zeros near \(g_n\), and the position of \(E_1(g_n)\) inside
the bad prime lobes.  In this section we discuss the theoretical mechanisms
behind these observations.  The observations themselves appear to be new, but
they are not detached from the standard analytic structure of the zeta
function. There are two closely
related mechanisms behind them.  First, the viscosity
\(\mu(g_n)=|Z'(g_n)/Z(g_n)|\) is naturally controlled by zeros lying close to
\(g_n\).  Second, by the explicit formula, local zero statistics are dual to
weighted prime-power sums at the same height. 

\subsection{Viscosity and local zero configurations}

Recall that
\[
        Z(t)=e^{i\theta(t)}\zeta\!\left(\frac12+it\right).
\]
At points where \(Z(t)\neq 0\), logarithmic differentiation gives
\[
        \frac{\zeta'}{\zeta}\!\left(\frac12+it\right)
        =
        -\theta'(t)-i\frac{Z'(t)}{Z(t)} .
\]
Thus the viscosity
\[
        \mu(t)
        =
        \left|\frac{Z'(t)}{Z(t)}\right|
\]
is the absolute value of the imaginary part of the logarithmic derivative of
\(\zeta\) on the critical line.

On the other hand, the logarithmic derivative is governed locally by the
nearby zeros.  In the range where the zeros on the critical line dominate the
local contribution, one has the heuristic approximation
\[
        \frac{Z'(t)}{Z(t)}
        \approx
        \sum_{\gamma}
        \frac{1}{t-\gamma}
        +
        \text{slowly varying terms},
\]
where the sum is effectively dominated by zeros \(\gamma\) close to \(t\).
Consequently, \(\mu(g_n)\) should be especially sensitive to the configuration
of the nearest zeros around the Gram point \(g_n\).

This is particularly transparent for isolated bad Gram points associated with
a confined pair of zeros.  Suppose that the two relevant zeros lie in the same
neighboring Gram interval on one side of \(g_n\).  If
\[
        \Delta_n
        =
        \frac{2\pi}{\log(g_n/2\pi)}
\]
denotes the local mean zero spacing, and if the two distances from \(g_n\) are
written in normalized form as
\[
        g_n-\gamma_1 = a_1\Delta_n,
        \qquad
        g_n-\gamma_2 = a_2\Delta_n,
\]
then the contribution of this pair to the normalized viscosity is of the form
\[
        \frac{2\pi}{\log(g_n/2\pi)}
        \left|
        \frac{1}{g_n-\gamma_1}
        +
        \frac{1}{g_n-\gamma_2}
        \right|
        =
        \left|
        \frac{1}{a_1}
        +
        \frac{1}{a_2}
        \right|.
\]
For a confined pair in a single neighboring Gram interval one expects
\(0<a_1,a_2\leq 1\), and therefore this contribution is at least \(2\), with
equality corresponding to the extremal case in which both zeros are pushed
toward the far endpoint of that interval.  This gives the basic intuition
behind the normalized lower scale
\[
        \frac{2\pi}{\log(g_n/2\pi)}\mu(g_n)
        \gtrsim 2
\]
for isolated bad Gram points of confined type.

It is important to note that this argument should not be read as a proof of the observed lower envelope.
It keeps only the contribution of the confined zero pair and does not yet
account for the nearby zeros outside the Gram block
\([g_{n-1},g_{n+1}]\), nor for the more distant zeros and slowly varying
background terms in the logarithmic derivative.  Nevertheless, it explains
why the viscosity is a natural quantity for detecting the local zero
configuration, and why extremal low-viscosity bad points should correspond to
highly coherent arrangements of the confined zero pair.
\subsection{Weighted prime sums and the logarithmic derivative}

We now turn to the second part of the mechanism: why should the prime cloud
remember information about the viscosity and the local zero configuration?
The point is that the prime phases appearing in \(E_1(g_n)\) are closely
related to the prime phases appearing in the logarithmic derivative of
\(\zeta\). For \(\Re s>1\), the Euler product gives the absolutely convergent identity
\[
        -\frac{\zeta'}{\zeta}(s)
        =
        \sum_{m=1}^{\infty}
        \frac{\Lambda(m)}{m^s},
\]
where \(\Lambda(m)\) is the von Mangoldt function.  Formally substituting
\(s=\frac12+it\) suggests that \(Z'(t)/Z(t)\) should be related to the
imaginary part of a prime-power phase sum:
\[
        \frac{Z'(t)}{Z(t)}
        \approx
        -\Im
        \sum_{m\le X}
        \frac{\Lambda(m)}{\sqrt m}
        e^{it\log m}.
\]
Equivalently,
\[
        \frac{Z'(t)}{Z(t)}
        \approx
        -\sum_{m\le X}
        \frac{\Lambda(m)}{\sqrt m}
        \sin(t\log m).
\]
Since \(\Lambda(m)\) is supported on prime powers, this sum may be written as
\[
        \sum_{m\le X}
        \frac{\Lambda(m)}{\sqrt m}
        e^{it\log m}
        =
        \sum_{p^k\le X}
        \frac{\log p}{p^{k/2}}
        e^{ikt\log p}.
\]
Thus the logarithmic derivative naturally involves the same prime phases
\(e^{it\log p}\), but with logarithmic weights, prime-power contributions,
and a cutoff scale \(X\).

Of course, this formal Euler-product expression is not an absolutely
convergent identity on the critical line.  A correct version must be stated
with smoothing and with the corresponding zero terms.  In the spirit of
Selberg's formula, one obtains expressions of the form
\[
        \frac{\zeta'}{\zeta}(s)
        =
        -
        \sum_{m\le X}
        \frac{\Lambda_X(m)}{m^s}
        +
        \text{zero terms}
        +
        \text{gamma and trivial-zero terms},
        \qquad
        s=\frac12+it,
\]
where \(\Lambda_X(m)\) is a smoothed von Mangoldt weight.  Taking imaginary
parts gives the schematic relation
\[
        \frac{Z'(t)}{Z(t)}
        =
        -\Im
        \sum_{m\le X}
        \frac{\Lambda_X(m)}{\sqrt m}
        e^{it\log m}
        +
        \text{zero corrections}
        +
        \text{smooth error terms}.
\]
This is the prime-side counterpart of the zero-side fact that
\(Z'(t)/Z(t)\) is controlled by nearby zeros. The prime cloud used in this paper is simpler.  It is defined by
\[
        E_1(g_n)
        =
        \sum_{p\leq N(g_n)}
        \frac{e^{ig_n\log p}}{\sqrt p}.
\]
This is not the exact prime sum appearing in the logarithmic derivative:
it omits the logarithmic weight \(\log p\), ignores prime powers \(p^k\) with
\(k\ge2\), and uses the AFE cutoff \(N(g_n)\) rather than a smoothed
explicit-formula cutoff.  Nevertheless, it is built from the same fundamental
phases
\[
        e^{ig_n\log p}.
\]
Thus \(E_1(g_n)\) may be viewed as an unweighted prime-only simplified projection of the
weighted prime-power data which appears naturally in the logarithmic
derivative.

This gives the conceptual bridge.  The relation between \(\mu(g_n)\) and the
nearby zeros is natural because \(Z'(g_n)/Z(g_n)\) is a logarithmic derivative.
The relation between nearby zeros and weighted prime-power sums is natural
because of the explicit formula.  What is not automatic, and what is isolated
empirically in this paper, is that the much simpler unweighted prime cloud
\(\mathcal P_N\), built from \(E_1(g_n)\), still retains a visible trace of
the same local zero geometry.

\section{Summary and Concluding Remarks}
\label{s:summary}
The Riemann zeta function has a long experimental tradition, beginning
with Riemann's unpublished investigations, following \cite{Riemann1859}. Gram points have
played a central role in this tradition, from the original computations of
Gram and Hutchinson \cite{Gram1903,Hutchinson} to the work of Turing,
Lehmer, and Rosser, Yohe, and Schoenfeld
\cite{Lehmer1956,RosserYoheSchoenfeld,Turing1953}.
Odlyzko's investigations of zero spacings demonstrated how numerical
experiments can reveal statistical structure beyond the verification
of individual zeros \cite{Odlyzko1987}, while Shanker's study of the
good-to-bad Gram point ratio continued this into an initial statistical exploration
at the level of Gram's law \cite{ShankerGramRatio}.

 In this work, we have introduced and experimentally studied the full AFE cloud \(\mathcal C_N\)
and its prime component \(\mathcal P_N\), as defined in \eqref{eq:clouds}, asking how good and bad
Gram points differ in their geometry. While the separation in the
full cloud follows naturally from the AFE, the pronounced bad lobes
in the prime cloud lead to the discovery of a more surprising structure. We have
then investigated how position within these lobes relates to the
viscosity
\[
        \mu(g_n)=\left|\frac{Z'(g_n)}{Z(g_n)}\right|
\]
and to the configuration of nearby zeros. 
Our experiments show
a close connection between these three features: prime-cloud
geometry, viscosity, and local zero configurations. Theoretically, although the explicit
formula couples zeros to weighted prime-power sums, it does not by
itself explain the bad lobes in \(\mathcal P_N\) or their internal
geometric ordering. Understanding why such information survives in
the simpler prime cloud, remains an open problem.

Finally, although the variational
approach developed in  \cite{JerbyDiscriminant} did not initially make its arithmetic connection to the zeros
of \(Z(t)\) explicit, the present experiments reveal a substantial
connection with prime phases and the arithmetic structure of the
zeta function. Understanding this connection may help bring the
geometry of the variation space and the arithmetic of primes into
a common framework. In particular, 
 the original motivation for
studying viscosity $\mu (g_n)$ of bad Gram points was to seek decompositions of the variation space
\(\mathcal Z_N\) in which the discriminant admits a natural expression.
It would be interesting to investigate whether the prime interpretation
developed here can guide such a decomposition.

\end{document}